\documentclass{amsart}

\usepackage{color,amscd,mathrsfs,fancyhdr}
\usepackage[all]{xy}

\usepackage{graphicx}

\newtheorem{theorem}{Theorem}[section]
\newtheorem{lemma}[theorem]{Lemma}
\newtheorem{prop}[theorem]{Proposition}

\newtheorem{conjecture}[theorem]{Conjecture}

\theoremstyle{definition}
\newtheorem{definition}[theorem]{Definition}
\newtheorem{example}[theorem]{Example}

\newtheorem{shitu}{Situation}

\theoremstyle{remark}
\newtheorem{remark}[theorem]{Remark}

\numberwithin{equation}{section}

\newcommand{\bi}{{\bf i}}

\newcommand{\bC}{{\bf C}}

\newcommand{\bQ}{{\bf Q}}
\newcommand{\bR}{{\bf R}}

\newcommand{\bZ}{{\bf Z}}

\newcommand{\fb}{{\mathfrak b}}

\newcommand{\fg}{{\mathfrak g}}

\newcommand{\fm}{{\mathfrak m}}

\newcommand{\ft}{{\mathfrak t}}

\newcommand{\fI}{{\mathfrak I}}

\newcommand{\fL}{{\mathfrak L}}

\newcommand{\fU}{{\mathfrak U}}

\newcommand{\Z}{\bZ}
\newcommand{\Q}{\bQ}
\newcommand{\R}{\bR}
\newcommand{\C}{\bC}

\newcommand{\su}{\mathfrak{su}}

\newcommand{\SO}{{\rm SO}}

\newcommand{\SU}{{\rm SU}}

\renewcommand{\epsilon}{\varepsilon}

\def\({\mathopen{}\left(}
\def\){\right)\mathclose{}}
\def\<{\mathopen{}\left<}
\def\>{\right>\mathclose{}}

\usepackage{multicol, color}

\definecolor{gold}{rgb}{0.85,.66,0}
\definecolor{cherry}{rgb}{0.9,.1,.2}
\definecolor{burgundy}{rgb}{0.8,.2,.2}
\definecolor{orangered}{rgb}{0.85,.3,0}
\definecolor{orange}{rgb}{0.85,.4,0}
\definecolor{olive}{rgb}{.45,.4,0}
\definecolor{lime}{rgb}{.6,.9,0}
\definecolor{green}{rgb}{.2,.7,0}
\definecolor{grey}{rgb}{.4,.4,.2}
\definecolor{brown}{rgb}{.4,.3,.1}

\def\makeautorefname#1#2{\AtBeginDocument{\expandafter\def\csname#1autorefname\endcsname{#2}}}

\begin{document}

\title{Atiyah-Floer Conjecture: a Formulation, a Strategy to Prove and Generalizations}

\author{Aliakbar Daemi}
\address{Simons Center for Geometry and Physics, State University of New York, Stony Brook, NY 11794}
\email{adaemi@scgp.stonybrook.edu}

\author{Kenji Fukaya}
\address{Simons Center for Geometry and Physics, State University of New York, Stony Brook, NY 11794 \& 
Center for Geometry and Physics, Institute for Basic Sciences (IBS), Pohang, Korea}
\email{kfukaya@scgp.stonybrook.edu}
\thanks{The second author was partially supported by NSF Grant No. 1406423 and the Simons Foundation through it Homological Mirror Symmetry Collaboration grant.}



\keywords{Differential geometry, Symplectic Topology, Low Dimensional Topology, Atiyah-Floer Conjecture}

\begin{abstract}
{Around 1988, Floer introduced two important theories: instanton Floer homology as invariants of 3-manifolds and Lagrangian Floer homology as invariants of pairs of Lagrangians in symplectic manifolds. Soon after that, Atiyah conjectured that the two theories should be related to each other and Lagrangian Floer homology of certain Lagrangians in the moduli space of flat connections on Riemann surfaces should recover instanton Floer homology. However, the space of flat connections on a Riemann surface is singular and the first step to address this conjecture is to make sense of Lagrangian Floer homology on this space. In this note, we formulate a possible approach to resolve this issue. A strategy to construct the desired isomorphism in the Atiyah-Floer conjecture is also sketched. We also use the language of $A_\infty$-categories to state generalizations of the Atiyah-Floer conjecture.}
\end{abstract}

\maketitle

\section{Introduction}
In a series of groundbreaking papers, Floer adapted methods of Morse homology to infinite dimensional settings and defined various homology theories. After the great success of Donaldson's work in 4-manifold topology, Floer employed Yang-Mills gauge theory to introduce an invariant of 3-manifolds, known as {\it instanton Floer homology} \cite{Fl:I}. In another direction, Floer built up on Gromov's theory of pseudo-holomorphic curves and defined {\it Lagrangian Floer homology} \cite{Fl:LHF}. The Atiyah-Floer conjecture \cite{At} concerns a relationship between instanton and Lagrangian Floer homologies.

Given an integral homology sphere $M$, instanton Floer homology determines an invariant of $M$, denoted by $I^*(M)$. This invariant is the homology of a chain complex $C_*(M)$ whose generators are identified with the non-trivial flat $\SU(2)$-connections on $M$. The differential of this chain complex is defined by an appropriate count of {\it anti-self-dual} connections on the 4-manifold $\R \times M$ which are asymptotic to flat connections on $M$. To follow this scheme more rigorously, flat $\SU(2)$-connections on $M$ are required to satisfy some non-degeneracy conditions. In general, one can avoid this issue by perturbing flat connections.

Lagrangian Floer homology is an invariant of a pair of Lagrangian submanifolds in a symplectic manifold. If $X$ is a symplectic manifold and $L_1$, $L_2$ are Lagrangian submanifolds of $M$, satisfying certain restrictive assumptions, then one can define Lagrangian Floer homology of $L_1$ and $L_2$, denoted by $HF(L_1,L_2)$. This invariant of $L_1$ and $L_2$ is the homology group of a chain complex $CF(L_1,L_2)$ whose set of generators can be identified with the intersection points of $L_1$ and $L_2$. The boundary operator for this chain complex is defined by an appropriate count of the number of pseudo-holomorphic strips which are bounded by $L_1$, $L_2$
and are asymptotic to given points in $L_1 \cap L_2$.

Let $\Sigma_g$ be a Riemann surface of genus $g\geq 1$. Then $R(\Sigma_g)$, the space of flat $\SU(2)$-connections on $\Sigma_g$ up to isomorphism, has a symplectic structure \cite{go}\footnote{To be more precise, the open dense subset of irreducible connections is a symplectic manifold.}. Suppose $M$ is an integral homology sphere with the following Heegaard splitting along the surface $\Sigma_g$:
\begin{equation}\label{handle-body}
M = H_g^0 \cup_{\Sigma_g}  H_g^1
\end{equation}
where $H^i_g$ is a handle-body of genus $g$. Then the space $R(H^i_g)$ of flat connections on $\Sigma_g$ which extend to flat connections of $H_g^i$ determines a `Lagrangian subspace' of $R(\Sigma_g)$. The Atiyah-Floer conjecture states that:
\begin{equation}\label{AF111}
I^*(M) \cong HF(R(H^1_g),R(H^2_g))
\end{equation}
where the right hand side is the Lagrangian Floer homology of $R(H^1_g)$ and $R(H^2_g)$, regarded as Lagrangians in $R(\Sigma_g)$. This Floer homology group is called the {\it symplectic instanton Floer homology} of $M$.


One of the main difficulties with the Atiyah-Floer conjecture is to give a rigorous definition of symplectic instanton Floer homology. The spaces $R(\Sigma_g)$, $R(H^1_g)$ and $R(H^2_g)$ are singular due to the existence of flat connections with non-trivial isotropy groups. One of the main goals of this article is to sketch the construction of a candidate for symplectic instanton Floer homology of integral homology spheres. Our method is mainly based on the construction of \cite{MW}. 

In \cite{MW}, Manolescu and Woodward replace $R(\Sigma_g)$ with another space $X$, which is a smooth closed manifold with a closed 2-form $\omega$. The 2-form $\omega$ is symplectic in the complement of a codimension 2 subspace $D\subset X$. For a Heegaard splitting as in \eqref{handle-body}, Lagrangian submanifolds $L_0,L_1$ of $ X\backslash D$ are also constructed in \cite{MW}. These Lagrangians are monotone in an appropriate sense.  A detailed study of the degenerate locus of $\omega$ in \cite{MW} allows Manolescu and Woodward to define Lagrangian Floer homology of $L_0$ and $L_1$. However, this version of symplectic instanton Floer homology is not expected to be isomorphic to $I^*(M)$. 

A variation of the triple $(X,\omega,D)$ leads to $(X_\theta,\omega_\theta,D)$ for $0<\theta<\frac{1}{2}$, where $\omega_\theta$ is a symplectic form on $X_\theta$. The submanifold  $D\subset X_\theta$ is a smooth divisor with respect to a K\"ahler structure in a neighborhood of $D$, compatible with the symplectic form $\omega_\theta$.  In Section \ref{sec:extendedmoduli}, we review the definition of $X$, $X_\theta$, which are respectively denoted by ${\widehat R}(\Sigma,p,\frac{1}{2})$, ${\widehat R}(\Sigma,p,\theta)$ and are constructed using the {\it extended moduli spaces of flat connections} \cite{Je,Hu}. 

For the Heegaard splitting in \eqref{handle-body}, we can also construct Lagrangian submanifolds $L_0, L_1\subset X_\theta\backslash D$. These Lagrangian manifolds are monotone only in the complement of $D$. Therefore, the general theory of \cite{Oh} or \cite{fooobook} is not strong enough to define Lagrangian Floer homology of $L_0$ and $L_1$. In Section \ref{sec:divisor complement}, we sketch the results of \cite{DF} on the construction of Floer homology of Lagrangians which are monotone in the complement of a divisor. There is also a group action of ${\rm PU}(2)$ on $X_\theta$, and the Lagrangians $L_0$ and $L_1$ are invariant with respect to this group action. Section \ref{sec:equivariant} concerns a review of equivariant Lagrangian Floer homology. In Section \ref{Sec:AFori}, we explain two different approaches that the action of ${\rm PU}(2)$ on $X_\theta$ can be combined with Floer homology of Lagrangians in divisor complements. In particular, we introduce two versions of symplectic instanton Floer homology $I^*_{symp}(M)$ and ${\overline I}^*_{symp}(M)$. We conjecture that $I^*_{symp}(M)$ is isomorphic to $I^*(M)$. This conjecture can be regarded as a rigorous version of the Atiyah-Floer conjecture. A possible approach to verify this version of the Atiyah-Floer conjecture is the content of Section \ref{modulispace}. There is a different approach to study the Atiyah-Floer conjecture proposed by Salamon. For this proposal and related works see \cite{Sa,Sawe, We, We2,Duncan}.

The Atiyah-Floer conjecture can be considered as a part of a broader program to define `instanton Floer homology of 3-manifolds with boundary'. Yang-Mills gauge theory in dimensions 2,3 and 4 is expected to produce the following structures:
\begin{enumerate}
\item
	For a 4-dimensional manifold and a ${\rm PU}(2)$-bundle, it defines a numerical invariant \`a la Donaldson.
\item
	For a 3-dimensional manifold and a ${\rm PU}(2)$-bundle, it defines a group which is instanton Floer homology.
\item
	For a 2-dimensional manifold and a ${\rm PU}(2)$-bundle, it defines a category.
\item
	The Donaldson invariants of a 4-manifold $N$ with boundary gives an element of the instanton Floer homology of the boundary of $N$.
\item
	Instanton Floer homology of a 3-dimensional manifold $M$ with boundary is expected to give an object of the category associated to $\partial M$.
\end{enumerate}
In 1992 and in a conference at the University of Warwick, Donaldson proposed a `first approximation' to the category associated to a Riemann surface $\Sigma$. In his proposal, an object of this category is a Lagrangian submanifold $L$ of the symplectic space $R(\Sigma)$. In this category, the set of morphisms between $L_1$ and $L_2$ should be defined as a version of Lagrangian Floer homology of $L_1$ and $L_2$. Then the Atiyah-Floer conjecture in \eqref{AF111} can be regarded as a consequence of the axioms of topological field theories. Of course, this plan faces with the same issue as before; the space $R(\Sigma)$ is singular and $HF(L_1,L_2)$ cannot be defined in a straightforward way. 

In Section \ref{Sec:conjectures}, we use the language of $A_\infty$-categories to state conjectures reagarding stronger versions of the above properties for instanton Floer homology in dimensions 2 and 3. In particular, we work with more arbitrary $G$-bundles on surfaces and 3-manifolds in this section. The conjectures of Section \ref{Sec:conjectures} are based on the second author's proposals in \cite{fu3}. In fact, his main motivation to introduce $A_\infty$-categories associated to Lagrangians in symplectic manifolds arises from studying this problem. In Section \ref{sec:generalgroup}, we will discuss how we expect that the extended moduli spaces of flat connections can be employed to produce these expected structure. The required results from symplectic geometry is discussed in Section \ref{Sec:ainfinity}. When working with a non-trivial $G$-bundle $\mathcal F$ over a Riemann surface, there are various cases that the moduli space $R(\Sigma,\mathcal F)$ of flat connections on $\mathcal F$ is smooth. For such choices of $\mathcal F$, there is an alternative approach to realize the above construction of 2- and 3-dimensional topological field theory based on Yang-Mills theory \cite{fu6,takagi,DFL}. The expected relationship between the two approaches is discussed in the last section of the paper.

In the present article, we only sketch the proofs of most theorems and a more detailed treatment will appear elsewhere. There are two main reasons for this style of writing. Firstly, these proofs contain some technical parts, and the limited available space in this proceedings journal does not allow us to include them in this article. Secondly, we hope that the current exposition together with the statements of the theorems and the conjectures gives the reader a better big picture of the projects related to this paper.

\subsubsection*{Acknowledgement.} We would like thank Max Lipyanskiy, Paul Seidel, Mohammad Tehrani, Michael Thaddeus and Aleksey Zinger for helpful discussions. We are very grateful to the support of the {\it Simons Center for Geometry and Physics}. We also thank the {\it IBS Center for Geometry and Physics} in Pohang for its hospitality during the authors' visit in the Summer of 2016.


\section{The Extended Moduli Spaces of Flat Connections}
\label{sec:extendedmoduli}

The most serious difficulty to define the symplectic counterpart of instanton Floer homology lies in the fact that the space of flat $\SU(2)$-connections on a Riemann surface is singular. In this section, we review Manolescu and Woodward's approach to go around this difficulty \cite{MW}. The main idea is to replace the moduli space of flat connections with the {\it extended moduli space of flat connections} \cite{Je,Hu}, which is a more nicely-behaved space. 
Fix a compact, simply connected and semi-simple Lie group $\widetilde G$, and let $G = \widetilde G/Z(\widetilde G)$ be the associated {\it adjoint group}. Here $Z(\widetilde G)$ denotes the center of $\widetilde G$, which is a finite group. The adjoint action of $\widetilde G$ on itself induces an action of $G$ on $\widetilde G$, which is denoted by $ad$. We will also write $\mathfrak g$ for the Lie algebra of $\widetilde G$ (or equivalently $G$). We identify $\fg$ with the dual of the Lie algebra $\fg^*$ using the inner product $-B(\cdot,\cdot)$ where $B(\cdot,\cdot)$  is the Killing form.

Suppose $\Sigma$ is an oriented and connected surface with genus $g \geq 1$ and $\mathcal F$ is a principal $G$-bundle over $\Sigma$. Define the {\it gauge group} $\mathcal G(\Sigma,\mathcal F)$ to be the space of all smooth sections of the bundle $\mathcal F \times_{ad} \tilde G$. Any element of the gauge group induces an automorphism of the bundle $P$. Fix a base point $p\in \Sigma$, an open neighborhood $U$ of $p$ and a trivialization of the fiber of $\mathcal F$ over $p$. Then the {\it based gauge group} is defined to be:
\begin{equation}
	\mathcal G_0(\Sigma,\mathcal F,p)= \{ g \mid g \in \mathcal G(\Sigma,\mathcal F),\,\,\, 
	g(q)= [e,\widetilde e]\in \mathcal F|_q \times_{ad} \tilde G \text{ for $q\in V\subseteq U$}\},
\end{equation}
where $e$ and $\widetilde e$ are respectively the unit elements of $G$ and $\widetilde G$ and $[e,\widetilde e]$ is the element of $\mathcal F|_q \times_{ad} \tilde G$ which is induced by the trivialization of $\mathcal F|_{q}$. Here $V$ is another open neighborhood of $q$ which is contained in $U$.

Let $\mathcal A_{\rm fl}(\Sigma,\mathcal F)$ denote the space of all flat connections on $\mathcal F$. Since the elements of $\mathcal G(\Sigma,\mathcal F)$ induce isomorphisms of $P$, this gauge group acts on $\mathcal A_{\rm fl}(\Sigma,\mathcal F)$. We will write $R(\Sigma,\mathcal F)$ for the quotient space. This quotient space has singular points in general. Away from the singular points, $R(\Sigma,\mathcal F)$ admits a symplectic structure \cite{go}. In the next definition, we recall the definition of a larger space which is a smooth symplectic manifold and can be used to replace $R(\Sigma,\mathcal F)$ for our purposes:

\begin{definition}
	Suppose $\mathcal A_{\rm ex}(\Sigma,F,p)$ consists of the pairs $(A,\xi)$ where $A$ is a flat connection on 
	$\Sigma\backslash \{p\}$ and $\xi \in \mathfrak g$. We require that there is an open neighborhood $V$ of $p$, included in $U$, 
	such that the restriction of $A$ to $V$ is equal to $\xi ds$
	with respect to the chosen trivialization of $\mathcal F$. Here we assume that $U$ is identified with the punctured disc centered at the origin
	in the 2-dimensional Euclidean space and $ds$ represents the angular coordinate.
	The group $\mathcal G_0(\Sigma,\mathcal F,p)$ clearly acts on 
	$\mathcal A_{\rm ex}(\Sigma,F,p)$ and the quotient space is denoted by ${\widehat R}(\Sigma,\mathcal F,p)$, and is called
	the {\it extended moduli space of flat connections on $\mathcal F$}.
\end{definition}		

The action of the based gauge group fixes the $\fg$-valued component of the elements of $\mathcal A_{\rm ex}(\Sigma,F,p)$. Therefore, we can define a map:
\begin{equation}\label{form2424}
\mu :
{\widehat R}(\Sigma,\mathcal F,p)  \to \frak g
\end{equation}
by: 
\[
  \mu([A,\xi]) = \xi.
\]
There is also an action of $G$ on the space ${\widehat R}(\Sigma,\mathcal F,p)$. Given any element $g\in \widetilde G$, we can construct an element of $\mathcal G(\Sigma,\mathcal F,p)$ which is given by $g$ in a neighborhood of $p$ that is contained in $U$ and is extended arbitrarily to the rest of $\Sigma$. This element of the gauge group acts on ${\widehat R}(\Sigma,\mathcal F,p)$ and this action only depends on the image of $g$ in the adjoint group $G$.

\begin{prop}[\cite{Je,Hu}]\label{symp-ext}
	There exist a $G$-invariant neighborhood $\fU$ of $0$ in $\fg$ and a symplectic structure on 
	$\mu^{-1}(\fU)\subset {\widehat R}(\Sigma,\mathcal F,p)$ 
	such that the action of $G$ is Hamiltonian and $\mu$ is the moment map.
	Furthermore, we have $R(\Sigma,\mathcal F)= \mu^{-1}(0)/G$, i.e., 
	$R(\Sigma,\mathcal F)$ can be identified with the symplectic quotient ${\widehat R}(\Sigma,\mathcal F,p)/\!\!/G$.
\end{prop}

\begin{remark}
	The extended moduli space for a disconnected $\Sigma$ is defined to be the product of the extended moduli spaces associated
	to the connected components of $\Sigma$. To be more precise, we need to fix a base point $p_i$ for each connected component
	$\Sigma_i $ of $\Sigma$. These base points altogether are denoted by $p$. Then:
	\[
	{\widehat R}(\Sigma,\mathcal F,p):=\prod_{i}{\widehat R}(\Sigma_i,\mathcal F_i,p_i)
	\]
	where $\mathcal F_i$ is the restriction of $\mathcal F$ to $\Sigma_i$. Note that there is an action of $G^m$ on the 
	moduli space in this case where $m$ is the number of connected components.
\end{remark}

For the rest of this section, we assume that $\widetilde G=\SU(2)$, and $\mathcal F$ is the trivial ${\rm PU}(2)$-bundle. In this case, the corresponding extended moduli space will be denoted by  ${\widehat R}(\Sigma,p)$. Consider the following subspace of $\su(2)$, the Lie algebra of $\SU(2)$:
\begin{equation}\label{form44}
	\Delta(\theta_0)	 =
	\left\{ \xi \in \su(2)\mid \xi \in \mathcal O_\theta
	,\,\, \theta \in [0,\theta_0)
	\right\}
\end{equation}
 Here $\mathcal O_\theta$ is the adjoint orbit of the following element of $\su(2)$:
\[
  \left(
	\begin{matrix}
		2\pi \bi \theta  &0 \\
		0& -2\pi \bi \theta
	\end{matrix}
	\right)
\]
In Proposition \ref{symp-ext} and for the case that $\mathcal F$ is the trivial ${\rm PU}(2)$-bundle, we can assume that $\fU=\Delta(\theta_0)$ where $\theta_0 \leq \frac{1}{2}$ \cite{MW}.	

Let $Q:\su(2)\to \R^{\geq 0}$ be the map which assigns to $\xi$ in \eqref{form44} the value $\theta$. Define the function:
\[
  \overline{\mu}:=Q \circ \mu: {\widehat R}(\Sigma,p)\to [0,\infty).
\]
This function is smooth on the complement of $\mu^{-1}(0)$, and defines a Hamiltonian vector field for an $S^1$-action on $\mu^{-1}(\Delta(\frac{1}{2}))$. This $S^1$-action extends to $\mu^{-1}(\frac{1}{2})$. Consider the map:
\[
  \widetilde \mu: {\widehat R}(\Sigma,p) \times \C \to \R
\]
defined by $\widetilde \mu(x,z)=\overline \mu(x)+\frac{1}{2}|z|^2$. This map is also a moment map for a Hamiltonian $S^1$-action. For each $0<\theta\leq \frac{1}{2}$, we define:
\begin{equation} \label{ex}
  {\widehat R}(\Sigma,p,\theta):=\widetilde \mu^{-1}(\theta)/S^1
\end{equation}
This space is a smooth manifold and is called the {\it cut down extended moduli space}. The manifold ${\widehat R}(\Sigma,p,\theta)$ can be decomposed as the disjoint union of the following spaces:
\begin{equation} \label{decom-ex}
  \mu^{-1}(\Delta(\theta)) \hspace{1cm}\overline{\mu}^{-1}(\theta)/S^1
\end{equation}
This construction is an example of non-abelian symplectic cut \cite{Wo20}.

\begin{theorem}\label{thm3434}{\rm(\cite{MW})}
	For $0<\theta< \frac{1}{2}$, the closed manifold ${\widehat R}(\Sigma,p,\theta)$ admits a symplectic form. The codimension two
	submanifold ${(\overline \mu)}^{-1}(\theta)/S^1$ of ${\widehat R}(\Sigma,p,\theta)$ is a symplectic hypersurface and 
	${\widehat R}(\Sigma,p,\theta)$ is K\"ahler in a neighborhood of ${(\overline \mu)}^{-1}(\theta)/S^1$.
	The ${\rm PU}(2)$-action on $ \mu^{-1}(\Delta(\theta))\subset {\widehat R}(\Sigma,p,\theta)$ extends to 
	a Hamiltonian action on ${\widehat R}(\Sigma,p,\theta)$ whose moment map $\widehat \mu$ 
	is the smooth extension of the map $\mu$
	defined on $\mu^{-1}(\Delta(\theta))$. The symplectic quotient ${\widehat R}(\Sigma,p,\theta)/\!\!/{\rm PU}(2)$
	can be identified with the moduli space of flat $\SU(2)$-connections ${R}(\Sigma)$ on $\Sigma$.
\end{theorem}


\begin{proof}
	Let $\C$ be given the negative of the standard symplectic form and consider the induced product symplectic form on 
	${\widehat R}(\Sigma,p) \times \C$.
	The ${\rm PU(2)}$-action on ${\widehat R}(\Sigma,p)$ induces a Hamiltonian action on ${\widehat R}(\Sigma,p) \times \C$.
	The map $\widetilde \mu$ is also the moment map for an $S^1$-action on the complement of $\mu^{-1}(0)\times \C$ in 
	${\widehat R}(\Sigma,p) \times \C$ which commutes with the ${\rm PU}(2)$-action. 
	The manifold ${\widehat R}(\Sigma,p,\theta)$ is the symplectic quotient of ${\widehat R}(\Sigma,p) \times \C$ 
	with respect to the $S^1$-action. Therefore, ${\widehat R}(\Sigma,p,\theta)$ admits a natural symplectic structure and a 
	Hamiltonian ${\rm PU(2)}$-action whose moment map is denoted by $\widehat \mu$. 
	It is straightforward to check that this action is an extension of the ${\rm PU}(2)$-action on $\mu^{-1}(\Delta(\theta))$, 
	and the symplectic quotient $\widehat \mu^{-1}(0)/{\rm PU}(2)$ is symplectomorphic to ${R}(\Sigma)$.

	It is shown in \cite[Proposition 4.6]{MW} that the hypersurface $\overline{\mu}^{-1}(\theta)/S^1$ is symplectomorphic to:
	\begin{equation*}
		P(\Sigma,p,\theta) \times \mathcal O_{\theta}
	\end{equation*}
	where $P(\Sigma,p,\theta)$ is equal to the following symplectic quotient:
	\begin{equation*}
		P(\Sigma,p,\theta):=\mu^{-1}(\mathcal O_\theta)/{\rm PU}(2).
	\end{equation*}
	The adjoint orbit $ \mathcal O_{\theta}$ is a 2-sphere equipped with the symplectic form $2\theta\cdot  \omega_0$ with $\omega_0$ being
	the standard volume form on a 2-sphere. The symplectic manifold $P(\Sigma,p,\theta)$ can be identified with a moduli space
	of parabolic bundles, and hence it is K\"ahler. 
	Now the claim about the symplectic form in a neighborhood of $\overline{\mu}^{-1}(\theta)/S^1$
	is a consequence of the standard neighborhood theorems in symplectic geometry.	
\end{proof}

\begin{remark}
	Theorem \ref{thm3434} to some degree can be extended to the case that $\theta=\frac{1}{2}$. The smooth manifold 
	${\widehat R}(\Sigma,p,\frac{1}{2})$ admits a closed 2-form which is non-degenerate on $\mu^{-1}(\Delta(\frac{1}{2}))$.
	The ${\rm PU}(2)$-action also extends to this manifold and the map $\mu$ on $\mu^{-1}(\Delta(\frac{1}{2}))$ determines 
	an $\su(2)$-valued map on ${\widehat R}(\Sigma,p,\frac{1}{2})$ which is the `moment' map for the ${\rm PU}(2)$-action.
	Since the symplectic form of ${\widehat R}(\Sigma,p,\frac{1}{2})$ is degenerate, the moment map condition should be interpreted
	carefully, and we refer the reader to \cite{MW} for more details. 
\end{remark}

A handlebody $H_g$ of genus $g$ determines a Lagrangian submanifold of ${\widehat R}(\Sigma,p,\theta)$. Assume that the boundary of $H_g$ is identified with $\Sigma_g$ and $\mathcal U$ is an open neighborhood of $p$ in $H_g$ whose intersection with the boundary is equal to $U$. Let $\mathcal A_{\rm fl}(H_g,p)$ be the space of all flat connections on the trivial $\SU(2)$-bundle over $H_g$ which are equal to the trivial connection in a neighborhood of $p$ contained in $\mathcal U$. Then the analogue of the based gauge group for the trivial $\SU(2)$-bundle on $H_g$ acts on $\mathcal A_{\rm fl}(H_g,p)$, and the quotient space is denoted by ${\tilde R(H_g,p)}$. 

\begin{prop}[\cite{MW}] \label{handle-body-lag}
	The restriction to the boundary induces a smooth embedding of ${\tilde R(H_g,p)}$ into ${\widehat R}(\Sigma,p,\theta)$.
	This submanifold of ${\widehat R}(\Sigma,p,\theta)$ is Lagrangian and invariant with respect to the action of ${\rm PU}(2)$.
\end{prop}

In \cite{MW}, the notion of monotone Lagrangians are extended to appropriate families of manifolds with degenerate symplectic forms\footnote{For the definition of monotonicity of Lagrangians in non-degenerate symplectic manifolds see Definition \ref{monotone}.}. In particular, it is shown there that ${\tilde R(H_g,p)}$ gives rise to a monotone Lagrangian submanifold of ${\widehat R}(\Sigma,p,\frac{1}{2})$. Thus, given a Heegaard splitting $H^0_g \cup_{\Sigma} H^1_g$ of a 3-manifold $M$, we can construct two monotone Lagrangian submanifolds $L_0$ and $L_1$ of ${\widehat R}(\Sigma,p,\frac{1}{2})$. The degeneracies of the symplectic form on ${\widehat R}(\Sigma,p,\frac{1}{2})$ is also studied carefully in \cite{MW} and subsequently the Lagrangian Floer homology of $L_0$ and $L_1$ is defined. This Lagrangian Floer homology group is an invariant of the 3-manifold $M$. One might hope to use this 3-manifold invariant in the formulation of the Atiyah-Floer conjecture. However, this invariant is not isomorphic to instanton Floer homology even in the simplest case that $M$ is equal to $S^3$. The right candidate for the the Atiyah-Floer conjecture should incorporate the action of ${\rm PU}(2)$ and should be defined as an appropriate ${\rm PU}(2)$-equivariant Lagrangian Floer homology of the Lagrangians $L_0$ and $L_1$.

As an alternative approach, one can try to define Lagrangian Floer homology for the Lagrangians associated to ${\tilde R(H^0_g,p)}$ and ${\tilde R(H^1_g,p)}$ in the non-degenerate symplectic manifold ${\widehat R}(\Sigma,p,\theta)$ where $\theta<\frac{1}{2}$. Unfortunately, these Lagrangians are not monotone anymore and one cannot use the general construction of \cite{Oh} to define Lagrangian Floer homology. It is not even clear to the authors whether these Lagrangians are unobstructed in the sense of \cite{fooobook}. Nevertheless, the general construction of \cite{DF} shows that one can replace the monotonicity of ${\tilde R(H^i_g,p)}$ in ${\widehat R}(\Sigma,p,\theta)$ with a weaker condition. We shall review this construction in Section \ref{sec:divisor complement}. The advantage of this construction is that it avoids the detailed analysis of the degeneracy of the symplectic form on ${\widehat R}(\Sigma,p,\frac{1}{2})$. In particular, we believe this framework can be used to define Lagrangian Floer homology for other choices of the Lie group $\widetilde G$. It also seems that this construction can be combined more easily with equivariant Floer homology. We shall review the definition of two such equivariant theories in Section \ref{Sec:AFori}.

%


\section{Equivariant Lagrangian Floer Homology} \label{sec:equivariant}

Let $(X,\omega)$ be a symplectic manifold. For simplicity, we assume that $X$ is a spin manifold. For any Lagrangian submanifold $L$ of $X$, there is a homomorphism $\mu_L : H_2(X;L;\Z) \to \Z$ which is called the {\it Maslov index}. (See, for example, \cite[Definition 2.1.15]{fooobook}.) 

\begin{definition} \label{monotone}
	A Lagrangian submanifold $L$ is {\it monotone}, if there exists $c>0$ such that the following identity holds for 
	all $\beta \in \pi_2(X,L)$:
        \begin{equation}\label{ineq32}
        		c\cdot {\mu_L}(\beta) = \omega(\beta).
        \end{equation}
        	The {\it minimal Maslov number of $L$} is defined to be:
	\begin{equation*}
		\inf \{ \mu_L(\beta) \mid \beta \in \pi_2(X,L),\,\omega(\beta) >0 \}.
	\end{equation*}
\end{definition}


Following Floer's original definition \cite{Fl:LHF}, Oh constructed Lagrangian Floer homology for a pair $L_0$ and $L_1$ of monotone Lagrangians, if one of the following conditions holds \cite{Oh}:
\begin{enumerate}
	\item[(m.a)] The minimal Maslov numbers of $L_0$ and of $L_1$ are both strictly greater than 2.
	\item[(m.b)] The Lagrangian submanifold $L_1$  is Hamiltonian isotopic to $L_0$. 
\end{enumerate}

Lagrangian Floer homology can be enriched when there is a group action on the underlying symplectic manifold. Such constructions have been carried out in various ways in the literature. (See Remark \ref{works-on-equiv-HF}). Let a compact Lie group $G$ acts on $X$, preserving the symplectic structure $\omega$. We fix a $G$-equivariant almost complex structure $J$ which is compatible with $\omega$. Note that the space of all such almost complex structures is contractible because the set of all $G$-invariant Riemannian metrics is convex. In the following, $H_G^*(M)$ for a $G$-space $M$ denotes the $G$-equivariant cohomology of $M$ with coefficients in $\R$. In the case that, $M$ is just a point, this group is denoted by $H_G^*$. The group $H_G^*(M)$ has the structure of a module over $H_G^*$ \cite{Bo:equi-co}. 

\begin{theorem}\label{existequiv}
	Let $L_0$, $L_1$ be $G$-equivariant spin Lagrangian submanifolds of $X$. 
	Suppose they are both monotone and satisfy either {\rm (m.a)} or {\rm (m.b)}. 
	Then there is a $H_G^*$-module $HF_G(L_0,L_1)$, called $G$-equivariant Lagrangian Floer homology of 
	$L_0$ and $L_1$. 
	In the case that the intersection $L_0 \cap L_1$ is clean, there exists a spectral sequence whose $E_2$ page is 
	$H_G^*(L_0\cap L_1)$ and which converges to $HF_G(L_0,L_1)$.
\end{theorem}

Recall that two submanifolds $L_0$ and $L_1$ of a smooth manifold $M$ have {\it clean intersection}, if $N=L_0\cap L_1$ is a smooth submanifold of $M$ and for any $x\in N$, we have $T_xN=T_xL_0\cap T_xL_1$.

\begin{proof}[Sketch of the proof]
	We assume that the intersection $L_0\cap L_1$ is a disjoint union of finitely  many $G$-orbits $G\cdot p$ for $p \in \mathcal A$. A {\it pseudo-holomorphic strip} $u : \R \times [0,1] \to X$ is a map that satisfies the following Cauchy-Riemann equation:
\begin{equation} \label{CR}
	\partial_t u+J\partial_\tau u=0
\end{equation}
We are interested in the moduli space of pseudo-holomorphic maps $u$ which satisfy the following boundary condition:
\begin{equation}\label{form38}
\aligned
u(\R \times \{0\}) &\subset L_0, \qquad
&u(\R \times \{1\}) \subset L_1 \\
\lim_{t\to+\infty} u(t,\tau) &\in G\cdot p, \qquad
&\lim_{t\to-\infty} u(t,\tau) \in G\cdot q.
\endaligned
\end{equation}
We will denote the homology classes of all such maps by $H(p,q)$. For a fixed $\beta \in H(p,q)$, let $\overset{\circ}{\mathcal M}(p, q;\beta;L_0,L_1)$ be the moduli space of pseudo-holomorphic maps satisfying \eqref{form38} and representing $\beta$, where we identify two maps $u$ and $u'$ if  $u'(t,\tau) = u(t+t_0,\tau)$ for some $t_0 \in \R$. Note that $\overset{\circ}{\mathcal M}( p, q;\beta;L_0,L_1)$ is invariant with respect to the action of the the group $G$. We also assume that this space is cut-down transversely by Equation \eqref{CR}. This moduli space can be compactified to a cornered manifold ${\mathcal M}(p, q;\beta;L_0,L_1)$ using {\it stable map compactification} \cite[Definition 10.3]{FO}. Codimension one boundary components of this space can be identified with the union of the fiber products:
\begin{equation}\label{form3838}
{\mathcal M}(p, r;\beta_1;L_0,L_1)
\times_{G \cdot r}{\mathcal M}(r, q;\beta_2;L_0,L_1)
\end{equation}
where $r \in \mathcal A$ and $\beta_1 \#\beta_2 = \beta$. Here $\#:H(p,r)\times H(r,q)\to H(p,q)$ is the concatenation of homology classes. Monotonicity and (m.a) or (m.b) are the main ingredients to prove these claims about stable map compactification.

The classifying space $BG$ and the universal bundle $EG$ over $BG$ can be approximated by finite dimensional manifolds $BG(N)$, $EG(N)$. To be more precise, suppose $EG(N)$ is a principal $G$-bundle over a manifold $BG(N)$ such that the homotopy groups of $EG(N)$ vanish up to degree $N$. We consider the approximate Borel construction
${\mathcal M}(p,q;\beta;L_0,L_1)
\times_G EG(N)$.
Taking asymptotic value as $t \to \pm \infty$, we obtain two evaluation maps as below:
\begin{equation*}\label{form39}
\begin{CD}
G\cdot p \times_G EG(N)
@<{{\rm ev}_{-\infty}}<<
{\mathcal M}(p, q;\beta;L_0,L_1)
\times_G EG(N)
@>{{\rm ev}_{+\infty}}>> G\cdot q \times_G EG(N).
\end{CD}
\end{equation*}
If ${\rm ev}_{+\infty}$ is a submersion, then we can define an operator:
\begin{equation}\label{form311009}
d_{p,q;\beta} : \Omega^*(G\cdot p \times_G EG(N))
\to \Omega^*(G\cdot q \times_G EG(N))
\end{equation}
between the space of differential forms by:
\begin{equation}\label{form31100}
d_{p,q;\beta}(h) = ({\rm ev}_{+\infty})_!({\rm ev}_{-\infty}^*h),
\end{equation}
where $({\rm ev}_{+\infty})_!$ is integration along the fiber.
Characterization of codimension one boundary components in \eqref{form3838} implies that:
\begin{equation}\label{form311}
d \circ d_{p,q;\beta} \pm d_{p,q;\beta} \circ d
= \sum_r\sum_{\beta_1+\beta_2=\beta} \pm
  d_{p,r;\beta_1} \circ d_{r,q;\beta_2}.
\end{equation}
Here $d$ is the usual de Rham differential. Therefore, the map $\delta_N= d + \sum d_{p,q;\beta}$ defines a differential, {\it i.e.}, it satisfies $\delta_N \circ \delta_N = 0$. Taking the limit $N \to \infty$, we obtain the 
equivariant Floer homology as the limit. 

In general, it might be the case that ${\mathcal M}(p,q;\beta;L_0,L_1)$ is not a smooth manifold or ${\rm ev}_{+\infty}$ is not a submersion. Then we can use the theory of {\it Kuranishi structures} and {\it continuous family of perturbations} on Kuranishi spaces to prove the same conclusion. In fact, following \cite{FuFu5}, we obtain a $G$-equivariant Kuranishi structure on ${\mathcal M}(p,q;\beta;L_0,L_1)$ and hence a Kuranishi structure on  ${\mathcal M}(p, q;\beta;L_0,L_1)\times_G EG(N)$. Then we can define a system of perturbations on these Kuranishi structures which give rise to a map as in \eqref{form311009} between the spaces of differential forms that satisfy \eqref{form311}.
\end{proof}

The elements of the moduli space ${\mathcal M}(p,q;\beta;L_0,L_1)$ can be regarded as solutions of a Fredholm equation which is defined on an infinite dimensional space and takes values in another infinite dimensional space. Roughly speaking, a Kuranishi structure on this moduli space replaces these infinite dimensional spaces with the spaces of finite dimensions. To be a bit more detailed, a Kuranishi structure is a covering of the moduli space with {\it Kuranishi charts}. For a point $p$ in the moduli space, a Kuranishi chart in a neighborhood of $p$ is a quadruple $(V,E,s,\psi)$ such that $V$ is a manifold, $E$ is a vector bundle, $s$ is a section of $E$ and $\psi$ is a homeomorphism from $s^{-1}(0)$ to an open neighborhood of $p$ in the moduli space. In general, we might need to work in the case that $V$ and $E$ are orbifold and orbi-bundle. Another part of the data of a Kuranishi structure is the set of coordinate change maps which explain how to glue different Kuranishi charts together. In order to get smooth spaces, we need to perturb the zero sets of the sections of Kuranishi charts in a consistent way, and continuous family of perturbations give a systematic way to achieve this goal. For a more detailed definition of Kuranishi structures and continuous family of perturbations, we refer the reader to \cite[Definition A1.5]{fooobook} and \cite{foootech2}.


\begin{remark} \label{works-on-equiv-HF}
	In this section, we discussed an approach to equivariant Lagrangian Floer homology which is 
	given by applying the Borel construction {\it after} taking finite dimensional reduction. 
	This approach was proposed independently by the second author in \cite[Section 7]{fu-1}
	and Viterbo. There are alternative approaches to equivariant Lagrangian Floer homology which also use 
	Borel construction but avoid virtual techniques. These approaches give rise to similar results as Theorem \ref{existequiv} 
	under more restrictive assumptions. In the case that $G=\Z/2\Z$, Floer homology coupled with Morse homology on 
	$EG$ is used in \cite{SeiSmi:loc} by Seidel and Smith to define
	equivariant Lagrangian Floer homology . 
	More recently, Hendricks, Lipshitz and Sarkar employed homotopy theoretic methods to define Lagrangian Floer homology 
	in the presence of the action of a Lie group \cite{HLS:equiv-finite,HLS:equiv-Lie}. 
	There are also various other equivariant theories for other Floer homologies (see, for example, 
	\cite{Don:YM-Floer,KM:monopoles-3-man,AB:equiv-I}).
\end{remark}

\section{Lagrangian Floer Theory in a Smooth Divisor Complement}\label{sec:divisor complement}


Let $(X,\omega)$ be a compact 
symplctic manifold and $D$ be a codimension $2$ submanimfold. We assume that $(X,D)$ is a K\"ahler manifold in a neighborhood of $D$, and $D$ is a smooth divisor in this neighborhood. 
\begin{definition}
	Let $L_1$ and $L_2$ be compact subsets of $X\setminus D$.
	We say $L_1$ is {\it Hamiltonian isotopic} to $L_2$ relative to $D$ 
	if there exists a compactly supported time dependent Hamiltonian 
	$H : (X\backslash D) \times [0,1] \to \R$ so that the Hamiltonian diffeomorphism $\varphi : X \setminus D \to X \setminus D$
	generated by $H$ sends $L_1$ to $L_2$, that is, $\varphi(L_1) = L_2$.
\end{definition}
\begin{definition}
	We say $L \subset X \setminus D$ is {\it monotone} if 
	\eqref{ineq32} holds for $\beta \in H_2(X\setminus D,L)$.
	The minimal Maslov number of $L$ relative to $D$ is defined as:
	\[
	  \inf \{\mu_L(\beta) \mid \beta \in \pi_2(X\setminus D,L),\,\omega(\beta) >0 \}.
	\]	
\end{definition}

In general, $\Lambda_0^{R}$, 
                    	the universal Novikov ring with ground ring $R$, consists of formal sums
                    	$\sum_i c_i T^{\lambda_i}$ where $c_i \in R$, $\lambda_i \in \R_{\ge 0}$, $\lim_{i\to\infty} \lambda_i = +\infty$, and 
                    	$T$ is a formal parameter. Similarly, 
	$\Lambda^{R}$ consists of $\sum_i c_i T^{\lambda_i}$ where $c_i \in R$, $\lambda_i \in \R$, $\lim_{i\to\infty} \lambda_i = +\infty$.
If $R$ is a field then $\Lambda^{R}$ is also a field.

\begin{theorem}\label{mainthm}
{\rm (\cite{DF})}
	Let $L_0,L_1$ be compact, monotone and spin Lagrangian submanifolds of $X\setminus D$.	
	We assume that {\rm (m.a)} or {\rm (m.b)} holds for these Lagrangains.
	Then there is a vector space $HF(L_0,L_1;X\setminus D)$ over $\Lambda^\Q$ which is called 
	the Lagrangian Floer homology of $L_0$ and $L_1$ relative to $D$, and satisfies the following properties:
	\vspace{-5pt}
	\begin{enumerate}
	\item[(i)]
		If $L_0$ is transversal to $L_1$ then we have
		$$
		{\rm rank}_{\Lambda^\Q} HF(L_0,L_1;X\setminus D) \le \# (L_0 \cap L_1).
		$$
	\item[(ii)]
		If $L_i'$ is Hamiltonian isotopic to $L_i$ in $X \setminus D$ 
		for $i=0,1$ then
		$$
		HF(L_0,L_1;X\setminus D) \cong HF(L'_0,L'_1;X\setminus D)
		$$
	\item[(iii)] 
		If either $L_1 = L_0$ or $\pi_1(L_0) = \pi_1(L_1) =0$, then 
		we can take $\Q$ as a coefficient ring instead of the 
		Novikov filed $\Lambda^\Q$.
	\item[(iv)]
		If $L_0 = L_1 =L$ holds, then there exists a 
		spectral sequence whose $E^2$ page is  $H(L;\Q)$
		and which converges to $HF(L,L;X\setminus D)$. 	\end{enumerate}
\end{theorem}
\begin{remark}
	The main point in Theorem \ref{mainthm} is that 
	we do {\it not} assume $L_i$ is a monotone Lagrangian submanifold 
	in $X$, for $i=1,2$. The general theory of \cite{fooobook,fooobook2}
	says that there is an obstruction to define Floer homology $HF(L_0,L_1)$.
	The Floer homology $HF(L_0,L_1;X\setminus D)$ uses only holomorphic disks 
	which `do not intersect' $D$. Therefore, the situation is similar to
	monotone Lagrangian Floer homology due to Oh \cite{Oh}.
	If $X\setminus D$ is convex at infinity, then \cite{fooobook,fooobook2} imply 
	that we can define Floer homology $HF(L_0,L_1;X\setminus D)$ satisfying the properties mentioned in Theorem \ref{mainthm}.
	Note that in Theorem \ref{mainthm}, we do not impose any kind of convexity assumption for $X \setminus D$.
	The specialization of the construction of Theorem \ref{mainthm} to the case that $L_i$ 
	is exact and the homology class of each component of $D$ is proportional to the Poincar\'e dual 
	of $[\omega]$ is given in \cite{sh}.
	
\end{remark}

\begin{proof}[Sketch of the proof]
	We assume that $L_0$ is transversal to $L_1$. Let $p,q \in L_0 \cap L_1$. We consider the moduli space
	$\overset{\circ}{\mathcal M}(p,q;\beta;L_0,L_1)$ of pseudo-holomorphic maps to $X\backslash D$
	which satisfy \eqref{CR} and \eqref{form38} for $G=\{1\}$. Following Floer \cite{Fl:LHF} and Oh \cite{Oh} 
	(see also \cite[Chapter2]{fooobook}), we can define $HF(L_0,L_1;X\setminus D)$ 
	if we obtain a compactification ${\mathcal M}^{\rm RGW}(p,q;\beta;L_0,L_1)$ of our moduli space 
	$\overset{\circ}{\mathcal M}(p,q;\beta;L_0,L_1)$ with the following properties:
\begin{enumerate}\label{formula31}
	\item[(I)] The compactification ${\mathcal M}^{\rm RGW}(p,q;\beta;L_0,L_1)$ 
	 carries a Kuranishi structure with boundary and corner.
	\item[(II)] The codimension one boundary of this moduli space is identified with the union of 
		\begin{equation}\label{splitend} 
		{\mathcal M}^{\rm RGW}(p,r;\beta_1;L_0,L_1) \times 
		{\mathcal M}^{\rm RGW}(r,q;\beta_2;L_0,L_1)
	\end{equation}
	for various $r \in L_1 \cap L_2$ and $\beta_1,\beta_2$
	with $\beta_1 + \beta_2 = \beta$. The (virtual) dimension $d(\beta)$ of ${\mathcal M}^{\rm RGW}(p,q;\beta;L_0,L_1)$
	is determined by the homology class $\beta$ and satisfies $d(\beta) = d(\beta_1) + d(\beta_2) + 1$ 
	for the boundary component in \eqref{splitend}.
\end{enumerate}
	We fix a {\it multisection}\footnote{See \cite[Definition A1.21]{fooobook2})} (or equivalently a multivalued perturbation) 
	which is transversal to $0$ and which is compatible with the description of the boundary as in \eqref{splitend}.
	Note that transversality implies that its zero set is the empty set when the virtual dimension is negative.
	Therefore, the zero set is a finite set in the case that the virtual dimension is $0$.
	Assuming $d(\beta) = 0$, let $\# {\mathcal M}^{\rm RGW}(p,q;\beta;L_0,L_1)$ 
	be the number (counted with sign and multiplicity) of the
	points in the zero set of the perturbed moduli space. Then we define:
	\[
	  \partial [p] = \sum_{q,\beta} \# {\mathcal M}^{\rm RGW}(p,q;\beta;L_0,L_1)[q].
	\]
	Here the sum is taken over all $q \in L_0 \cap L_1$ and the homology classes $\beta$ such that $d(\beta) = 0$. 
	In case $L_0 = L_1$ or $\pi_1(L_0) = \pi_1(L_1) = 0$ the right hand side is a finite sum. Otherwise we use appropriate Novikov ring and put the 
	weight $T^{\omega(\beta)}$ to each of the 
	terms of the right hand side, so that the right hand side converges in $T$ adic topology.
	As it is customary with other Floer theories, we can show $\partial\circ \partial = 0$ 
	using the moduli spaces associated to the homology classes $\beta$
	 with $d(\beta) =1$. (See, for example, \cite{Fl:LHF,Oh}.) The proofs of parts $(ii)$ and $(iii)$) are also similar to the 
	 proof of the corresponding statements in the case of usual monotone Lagrangian Floer homology.
	 
	If $X\backslash D$ is convex at infinity, then we can let ${\mathcal M}^{\rm RGW}(p,q;\beta;L_0,L_1)$
	to be the closure of $\overset{\circ}{\mathcal M}(p,q;\beta;L_0,L_1)$ in the moduli space of stable holomorphic maps 
	to $X$. In this case, monotonicity can be used to show that \eqref{formula31} gives all the configurations
	 appearing in the boundary of ${\mathcal M}^{\rm RGW}(p,q;\beta;L_0,L_1)$.
	 
	 In case we do not assume monotonicity, disk bubble can occur as the other type 
	 of boundary components. (See \cite[Subsection 2.4.5]{fooobook} for example.)

	The stable map compactification in \cite[Subsection 7.1.4]{fooobook2} does {\it not} give a compactification 
	${\mathcal M}^{\rm RGW}(p,q;\beta;L_0,L_1)$ with the required properties.
	The issue is that in the stable map compactification a stable map with a sphere bubble which 
	is contained completely in the divisor $D$ is included.
	 At the point of such stable maps, 
	 the limit of the following two kinds of sequence of stable maps are mixed up.
	\begin{enumerate}
	\item[(A)] 
		A limit of a sequence of pseudo-holomorphic disks $u_i : (D^2,\partial D^2) \to (X,L)$
		such that $u_i(D^2) \cap D = \emptyset$.
	\item[(B)]
		A limit of a sequence of stable maps $u_i : (\Sigma_i,\partial \Sigma_i) \to (X,L)$,
		where $\Sigma_i$ is a disk plus sphere bubbles, 
		and such that $u_i(\Sigma_i) \cap D \ne \emptyset$.
	\end{enumerate}
	We need to include (A) in our moduli space but (B) is not supposed to be 
	an element of the moduli space.
	
As it is shown in  Figure \ref{zu0}, elements given as the limit points of type (A) and type (B) can be mixed up in the stable map compactification.
Here all the sphere bubbles in the figure are contained in $D$. The numbers written in the sphere bubbles $S^2$ are the intersection numbers $[S^2] \cap D$. The numbers written at the roots of the sphere bubble are the intersection multiplicities of the disk with the divisor $D$.
The configuration shown as (a) is a limit of disks as in (A) above since $2 + (-2) = 1 + (-1) =0$.
The configuration shown as (b) is {\it not} a limit of disks as in (A),since $2 + (-1) \ne 0 \ne 1 + (-2)$. However, these two configurations can intersect 
in the limit, which is the stable map shown as (c) in the figure.

Note that a limit of the configuration (b) in the figure can split into two pieces as shown in the figure.Then the union of the disk component together with sphere bubble rooted on it is {\it not} monotone. Thus if we include (b), then there will be a trouble to show \eqref{splitend}.

\begin{figure}[h]
\centering
\includegraphics[scale=0.3]{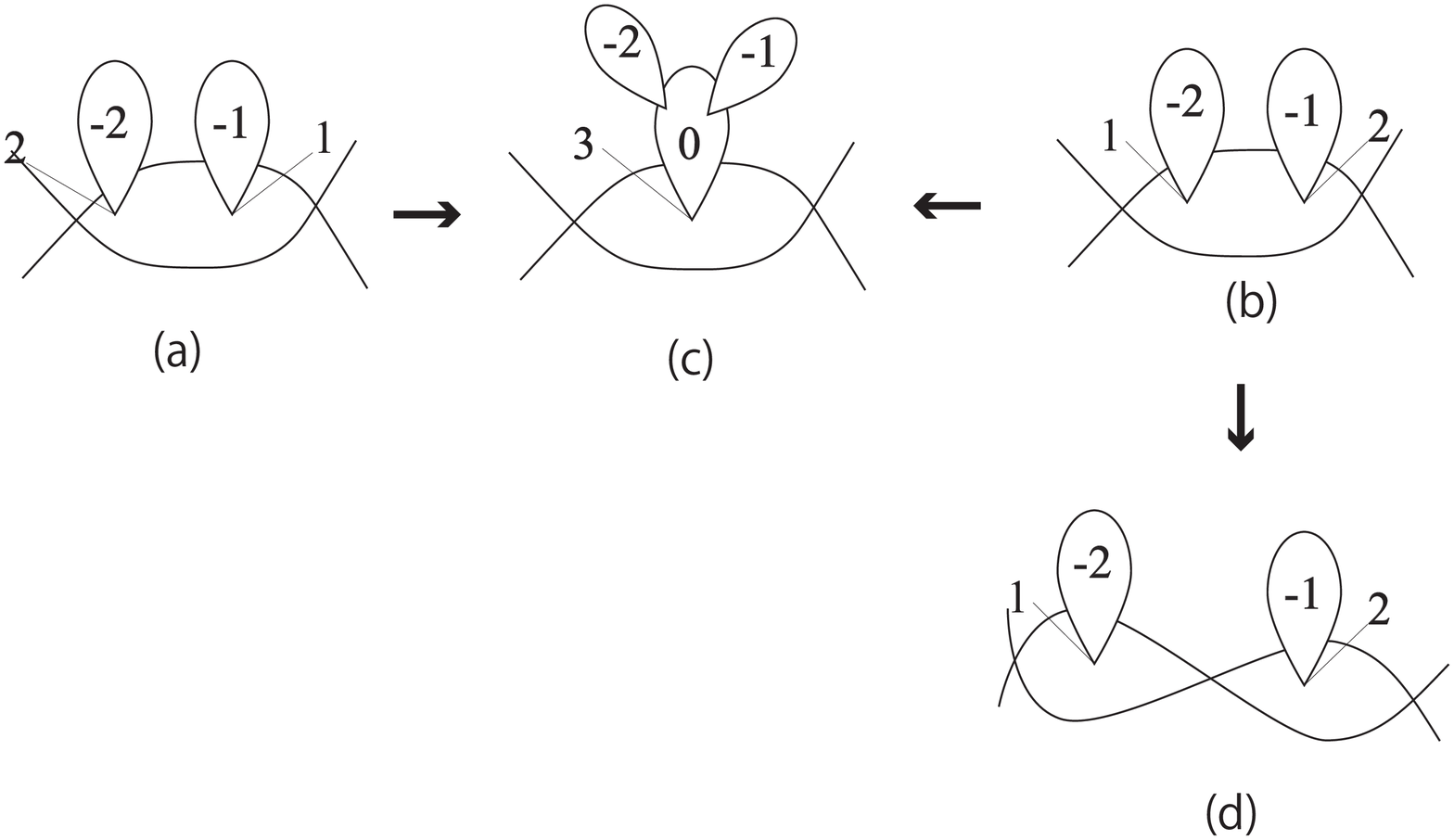}
\caption{(a) and (b) mixed up.}
\label{zu0}
\end{figure}


	The idea to resolve this issue is to use a compactification 
	which is different from the stable map compactification (in $X$).
	We use the compactification used in relative 
	Gromov-Witten theory, where the 
	limits of type (A) and type (B) are clearly separated. (See \cite{LR,IP,JL,JL2,groSie,pa,TZ}.)
	Namely in this compactification configuration (c) in the 
	figure comes with additional information so that the 
	limit of type (A) and of type (B) becomes a different element 
	in this compactifiation.
	Using this fact we can then show the above properties (I)(II).
\end{proof}

\section{The Atiyah-Floer Conjecture}
\label{Sec:AFori}

Floer's original instanton Floer homology is an invariant of 3-manifolds which have the same integral homology as the 3-dimensional sphere \cite{Fl:I}. Given an integral homology sphere $M$, he introduced a chain complex $C_*(M)$ with the differential $\partial$, whose homology is defined to be instanton Floer homology and is denoted by $I^*(M)$. The chain complex $C_*(M)$ is a free group generated by non-trivial flat $\SU(2)$-connections. Since any $\SU(2)$-bundle over a 3-manifold can be trivialized, these connections all have the same topological type. 

The differential $\partial$ is defined by considering the moduli space of instantons on the trivial $\SU(2)$-bundle $P$ over $\R \times M$. To be a bit more detailed, fix a product metric on $\R \times M$ corresponding to a fixed metric on $M$. The Hodge $*$-operator decomposes the space of 2-forms to {\it anti-self-dual} and {\it self-dual} forms. Then an instanton on $\R \times M$ is a connection $A$ on $P$ such that: 
\begin{equation} \label{instanton-eq}
	F^+(A)=0 \hspace{2cm} |\!|F(A)|\!|_2<\infty
\end{equation}
where $F^+(A)$ and $|\!|F(A)|\!|_2$ are respectively the self-dual part and the $L^2$ norm of the curvature of $A$. For any instanton $A$, there are flat connections $a_+$ and $a_-$ such that:
\begin{equation} \label{asymp}
	\lim_{t \to \pm \infty} A|_{\{t\}\times M}=a_{\pm}
\end{equation}
Translation in the $\R$-direction and $\SU(2)$-bundle automorphisms act on the space of instantons. The  quotient space of instantons satisfying \eqref{asymp} with respect to these two actions is denoted by $\mathcal M(a_-,a_+;M)$. Moreover, if we require $|\!|F(A)|\!|_2^2$ to be equal to a fixed real number $E$, then the resulting space is denoted by $\mathcal M(a_-,a_+;E;M)$. The differential $\partial(a)$ for a non-trivial flat connection $a$ is defined as:
\begin{equation*}
	\partial(a)=\sum \#\mathcal M(a,b; E;M)\cdot b
\end{equation*}
where the sum is over all $E$ and $b$ that $\mathcal M(a,b; E;M)$ is 0-dimensional. Here $\#\mathcal M(a,b; M)$ denotes the signed count of the points in the 0-dimensional space $\mathcal M(a,b;E; M)$. In general, we might need to perturb the equation in \eqref{instanton-eq} as the space of flat connections on $M$ and the space of instantons on $\R \times M$ might not be cut down by transversal equations.

There are several other versions of instanton Floer homology in the literature. The trivial connection on an integral homology sphere $M$ does not play any role in the definition of $I^*(M)$. An alternative version of this invariant, constructed in \cite{Don:YM-Floer}, uses the moduli spaces $\mathcal M(a,b,E;M)$ where $a$ or $b$ could be the trivial connection. We will write $\overline I^*(M)$ for this invariant, which is an $H_{{\rm PU}(2)}^*$-module\footnote{The original notation for this invariant in \cite{Don:YM-Floer} is $\overline{\overline{HF}}(M)$.}.

Next, we aim to construct a version of symplectic instanton Floer homology which is conjecturally isomorphic to $I^*(M)$ for an integral homology sphere $M$. We shall apply a combination of the constructions of the previous two sections to the following pair of a symplectic manifold and a smooth divisor, introduced in \eqref{ex} and \eqref{decom-ex}:
\begin{equation} \label{symp-div}
	(X,D):=({\widehat R}(\Sigma,p,\theta), \overline{\mu}^{-1}(\theta)/S^1)
\end{equation}
with $0<\theta < \frac{1}{2}$. This version of symplectic instanton Floer homology can be regarded as an equivariant version of a variation of the construction in \cite{MW}. (See Remark \ref{MW-AF}.) Fix a Heegaard splitting of the 3-manifold $M$:
\begin{equation}\label{form73}
	M = H^0_g \cup_{\Sigma_g} H^1_g.
\end{equation}
According to Proposition \ref{handle-body-lag}, we can form the Lagrangian submanifolds $\tilde R(H^i_g,p)$ of $X$ associated to this Heegaard splitting. The following Lemma about the intersection of these Lagrangians can be proved using {\it holonomy perturbations} \cite{Cliff:Cas-invt,Fl:I,Don:YM-Floer,He}. We omit the details here:

\begin{lemma}\label{lem7474}
	There are Hamiltonian isotopies of the Lagrangians $\tilde R(H^i_g,p)$ in $X\backslash D$ to submanifolds with clean intersection.
	Moreover, we can assume that each connected component of the intersection of the perturbed Lagrangians is
	either a point which consists of the trivial connection or a single ${\rm PU}(2)$-orbit.
\end{lemma}

Suppose $L_i$ denotes the perturbation of the Lagrangian $\tilde R(H^i_g,p)$ provided by Lemma \ref{lem7474}. Since $\tilde R(H^i_g,p)$ is monotone in $X\backslash D$ \cite{MW}, the Lagrangian $L_i$ is also monotone in $X\backslash D$. The manifold $L_i$ is diffeomorphic to the Cartesian product of $g$ copies of $\SU(2)$ \cite{MW}. In particular, it can be equipped with a spin structure. The intersection of $L_0$ and $L_1$ can be decomposed as:
\begin{equation}\label{decompLL}
	L_0 \cap L_1=\{\theta\} \cup \bigcup_{a \in \mathcal A} R_{a}
\end{equation}
where $R_a \cong {\rm PU}(2)$. Here $\theta$ denotes the trivial connection. Let $\mathcal A_+ = \mathcal A \cup \{\theta\}$ and $R_\theta= \{\theta\}$.

For $a,b \in \mathcal A$, define $\overset{\circ}{\mathcal M}(a,b;\beta;L_0,L_1)$ to be the moduli space of maps $u : \R \times [0,1] \to X \setminus D$ which satisfy the analogues of \eqref{CR}, \eqref{form38} and represent the homology class $\beta \in H(a,b)$. As before, we also identify two maps $u$ and $u'$ if  $u'(\tau,t) = u(\tau+\tau_0,t)$ for some $\tau_0 \in \R$. There is an obvious ${\rm PU}(2)$ action on this moduli space. We can also form the restriction maps:
\begin{equation}\label{form7474}
	{\rm ev}_{-\infty} : \overset{\circ}{\mathcal M}(a,b;\beta;L_0,L_1) \to R_a,\quad
	{\rm ev}_{+\infty} : \overset{\circ}{\mathcal M}(a,b;\beta;L_0,L_1)\to R_b.
\end{equation}
A combination of the proofs of Theorems  \ref{existequiv} and \ref{mainthm} can be used to prove the following Proposition: 
\begin{prop}\label{prop7575}
	There exists a compactification of $\overset{\circ}{\mathcal M}(a,b;\beta;L_0,L_1)$, denoted by ${\mathcal M}^{\rm RGW}(a,b;\beta;L_0,L_1)$,
	which satisfies the following properties:
	\vspace{-5pt}
\begin{enumerate}
	\item[(i)] 
	This space has a Kuranishi structure with corners. The ${\rm PU}(2)$ action of $\overset{\circ}{\mathcal M}(a,b;\beta;L_0,L_1)$ 
	extends to ${\mathcal M}^{\rm RGW}(a,b;\beta;L_0,L_1)$ and the Kuranishi structure is ${\rm PU}(2)$-equivariant.
	The evaluation maps in \eqref{form7474} also extend to ${\mathcal M}^{\rm RGW}(a,b;\beta;L_0,L_1)$ 
	and are underlying maps of 
	${\rm PU}(2)$-equivariant weakly submersive maps.\footnote{See \cite[Definition 32.1]{fooospectr} for its definition.}	
	\item[(ii)]
	Let $d(\beta)$ be the virtual dimension of ${\mathcal M}^{\rm RGW}(a,b;\beta;L_0,L_1)$. 
	For any $d$, there are only finitely many choices of $\beta$ such that 
	${\mathcal M}^{\rm RGW}(a,b;\beta;L_0,L_1)$ is nonempty and $d(\beta) = d$.
	There also exists ${\rm deg}:\mathcal A \to \Z/8\Z$ such that $\deg_+(\theta)=0$ and for any $a \in \mathcal A$,
	$b \in \mathcal A_+$, $\beta \in H(a,b)$,
	we have:
	\begin{equation} \label{mod-8-1}
	  d(\beta) \equiv {\rm deg}(b) - {\rm deg}(a) + 2 \mod 8.
	\end{equation}
	Moreover, if $b \in \mathcal A_+$ and $\beta \in H(\theta,b)$ then:
	\begin{equation}\label{mod-8-2}
	  d(\beta) \equiv{\rm deg}(b)  - 1 \mod 8.
	\end{equation}
	\item [(iii)]
	The codimension one boundary components of ${\mathcal M}^{\rm RGW}(a,b;\beta;L_0,L_1)$ consists of fiber products
	\begin{equation}\label{form7575}
		{\mathcal M}^{\rm RGW}(a,c;\beta_1;L_0,L_1) \times_{R_c} {\mathcal M}^{\rm RGW}(c,b;\beta_2;L_0,L_1),
	\end{equation}
	where the union is taken over $c \in \mathcal A_+$ and $\beta_1 \in H(a,c),\beta_2 \in H(c,b)$ with $\beta_1 \#\beta_2 = \beta$.
\end{enumerate}
\end{prop}

\begin{remark}
	Characterization of co-dimension one boundary components in \eqref{form7575} implies that if $c \ne \theta$, then:
	\[
	  d(\beta_1 \#\beta_2) = d(\beta_1) +  d(\beta_2) - 3
	\]
	 and if $c = \theta$, then:
	\[
	  d(\beta_1 \#\beta_2) = d(\beta_1) +  d(\beta_2).
	\]
	This is consistent with the identities in \eqref{mod-8-1} and \eqref{mod-8-2}.
\end{remark}

Analogous to the construction of Section \ref{sec:equivariant}, we can use the compactification provided by this proposition to define a Lagrangian Floer homology group :
\begin{equation} \label{symp-equiv}
	HF_{{\rm PU}(2)}(L_0,L_1;X \setminus D)
\end{equation}
for an integral homology sphere. This Lagrangian Floer homology group is a module over $H^*_{{\rm PU}(2)}= H^*_{{\SU}(2)}$. The following conjecuture states that this module is a 3-manifold invariant. This invariant can be regarded as a version of symplectic instanton Floer homology and is denoted by $\overline I^*_{symp}(M)$. 
\begin{conjecture} \label{symp-equiv}
	The $H^*_{{\rm PU}(2)}$-modules in \eqref{symp-equiv} for different choices of Heegaard splitting 
	are isomorphic to each other.
\end{conjecture}
\begin{remark} \label{inv}
	We hope to address Conjecture \ref{symp-equiv} in the same way as in the proof of the corresponding result in \cite{MW}. 
	(The result of \cite{MW} can be regarded as a non-equivariant version of Conjecture \ref{symp-equiv}.) 
	Following the arguments in \cite{MW} requires us to consider quilted Floer homology of Lagrangian correspondences such that each 
	Lagrangian correspondence is from a pair $(X_1,D_1)$ of a symplectic manifold and a smooth divisor to another such pair $(X_2,D_2)$. 
	Consequently, we need to study the moduli space of holomorphic curves for the pair $(X_1\times X_2, (D_1 \times X_2) \cup (X_1 \times D_2))$. 
	
	The space $(D_1 \times X_2) \cup (X_1 \times D_2)$ is a normal crossing divisor in $X_1\times X_2$. 
	The extension of the theory of Section \ref{sec:divisor complement} 
	to normal crossing divisors is the content of a work in progress and its details have not been completely worked out yet. 
	(See Conjecture \ref{cong59} and Remark \ref{normal-crossing}.) 
	However, as it is explained in \cite[Section 12]{FuFu50}, 
	we can use a different compactification of holomorphic discs whose target is the product $X_1\times X_2$. 
	This compactification is denoted by $\mathcal M^{\prime}$ and is discussed in \cite[Section 12]{FuFu50}. 
	In this compactification, the sphere bubbles on two factors are studied separately.
	It is plausible that adapting this construction to our set up allows us to avoid the case of holomorphic discs
	 in the complement of normal crossing divisors and to work only with smooth divisors.
\end{remark}

There is an alternative version of symplectic instanton Floer homology constructed by the moduli spaces ${\mathcal M}^{\rm RGW}(a,b;\beta;L_0,L_1)$. The ${\rm PU}(2)$ action on the space ${\mathcal M}^{\rm RGW}(a,b;\beta;L_0,L_1)$
is free unless $a=b=\theta$. If $a=b=\theta$, then the action is still free unless $\beta=0$, which is the homology class of the constant map. The moduli space
${\mathcal M}^{\rm RGW}(\theta,\theta;0;L_0,L_1)$ consists of a single element. Therefore, the quotient space
\[
  \overline{\mathcal M}^{\rm RGW}(a,b;\beta;L_0,L_1):={\mathcal M}^{\rm RGW}(a,b;\beta;L_0,L_1)/{\rm PU}(2)
\]
has an induced Kuranishi structure. Proposition \ref{prop7575} can be used to verify the following lemma:
\begin{lemma}\label{lem6767}
	For $a,b \in \mathcal A$, the boundary of $\overline{\mathcal M}^{\rm RGW}(a,b;\beta;L_0,L_1)$ 
	is the union of two types of spaces:
	\begin{enumerate}
	\item The direct product:
		\[
		  \overline{\mathcal M}^{\rm RGW}(a,c;\beta_1;L_0,L_1) \times 
		  \overline{\mathcal M}^{\rm RGW}(c,b;\beta_2;L_0,L_1)
		\]
		for $c \in \mathcal A$, $\beta_1 \in H(a,c)$ and $\beta_2 \in H(c,b)$ such that 
		$\beta_1 \#\beta_2 = \beta$.
	\item The quotient of the union of direct products 
		\[
		  {\mathcal M}^{\rm RGW}(a,\theta;\beta_1;L_0,L_1) \times 
		  {\mathcal M}^{\rm RGW}(\theta,b;\beta_2;L_0,L_1)
		\]
		by the diagonal ${\rm PU}(2)$ action. 
		Here the union is taken over $\beta_1 \in H(a,\theta)$ and $\beta_2 \in H(\theta,b)$ 
		with $\beta_1 \#\beta_2 = \beta$.
	\end{enumerate}
\end{lemma}

We pick a system of ${\rm PU}(2)$ invariant multi-sections over each moduli space ${\mathcal M}^{\rm RGW}(a,b;\beta;L_0,L_1)$ that is compatible with the description of the boundaries in \eqref{form7575}.
This is equivalent to choosing a system of 
multi-sections over various 
$\overline{\mathcal M}^{\rm RGW}(a,b;\beta;L_0,L_1)$
that is compatible with the 
description of the boundaries in Lemma \ref{lem6767}. In the case that $d(\beta)=0$ and $a,b \ne \theta$, Lemma  \ref{lem6767} and the compatibility of the multi-sections show that the zero set of the multisection 
in the moduli space $\overline{\mathcal M}^{\rm RGW}(a,b;\beta;L_0,L_1)$ is a compact 0-dimensional space. Therefore, we can count the number of points in this space (with signs) to define:
\begin{equation}\label{numberRGWkado}
	\# \overline{\mathcal M}^{\rm RGW}(a,b;\beta;L_0,L_1).
\end{equation}

Now we are ready to define another version of symplectic instanton Floer homology for integral homology spheres. Define:
\begin{equation} \label{symp-chain-cx-1}
  C^*_{symp}(M):= \bigoplus_{a \in \mathcal A} \Q [a]
\end{equation}
and 
\begin{equation}\label{symp-chain-cx-2}
  \partial [a]:=\sum_{b \in \mathcal A, \beta\in H(a,b)}\left( \# \overline{\mathcal M}^{\rm RGW}(a,b;\beta;L_0,L_1)\right) [b].
\end{equation}
where the sum is over all $b \in \mathcal A$ and $\beta\in H(a,b)$ that $d(\beta)=0$. Another application of Lemma \ref{lem6767} and the compatibility of the multi-sections show that $\partial^2=0$. To be a bit more detailed, the terms in $\partial^2(a)$, for a non-trivial flat connection $a$, are in correspondence with the boundary points of the 1-dimensional moduli spaces $\overline{\mathcal M}^{\rm RGW}(a,b;\beta;L_0,L_1)$ which are of type (1) in Lemma \ref{numberRGWkado} For a 1-dimensional moduli space, the space of boundary points of type (2) is empty, because each component of the space of type (2) boundary points has dimension at least 3.

The homology of the chain complex in \eqref{symp-chain-cx-1} and \eqref{symp-chain-cx-2} is denoted by $I^*_{symp}(M)$. Note that our definition of $I^*_{symp}(M)$ resembles Floer's instanton homology $I^*(M)$ in the sense that the trivial connection $\theta$ doss not enter into the definition of the corresponding chain complex. The following is the analogue of Conjecture \ref{symp-equiv}. The same comment as in Remark \ref{inv} applies to this conjecture.
\begin{conjecture} \label{symp-equiv-2}
	The group $I^*_{symp}(M)$ is an invariant of the integral homology sphere $M$.
	That is to say, the homology of the chain complex $(C^*_{symp}(M),\partial)$
	is independent of the choices of the Heegaard splitting.
\end{conjecture}

\begin{remark}
	In the course of defining $I^*_{symp}(M)$, we only need the moduli spaces of virtual dimension $1$ or $0$.
	Therefore, we do not need to prove the smoothness of the coordinate change maps 
	of our Kuranishi structure.
	We also do not need to study triangulation of the zero set 
	of our multi-sections. For example, we can discuss in the same way as in
	\cite[Section 14]{foootech2}.
\end{remark}

The first part of the following conjecture can be regarded as a rigorous formulation of the original version of the Atiyah-Floer conjecture for integral homology spheres. In Section \ref{modulispace}, we sketch a plan for the proof of the first part of Conjecture \ref{conj7878}.
\begin{conjecture}\label{conj7878}
	For any integral homology sphere $M$, the vector spaces 
	$I^*(M)$ and $I^*_{symp}(M)$ are isomorphic to each other. 
	The $H^*_{{\rm PU}(2)}$-modules $\overline I^*(M)$ and $\overline I_{symp}^*(M)$ 
	are also isomorphic to each other.
\end{conjecture}

\begin{remark} \label{MW-AF}
	One can forget the ${\rm PU}(2)$-action on ${\mathcal M}^{\rm RGW}(a,b;\beta;L_0,L_1)$ and 
	apply the construction of the previous section to define (non-equivariant) Lagrangian Floer homology 
	for the Lagrangians $L_0$ and $L_1$ in the complement of $D$. The resulting Floer homology is essentially 
	the same 3-manifold invariant as the version of symplectic instanton Floer homology that is constructed in \cite{MW}.
	There is also an analogue of the Atiyah-Floer conjecture for this invariant.
	It is conjectured in \cite{MW} that this invariant is isomorphic to an alternative version of instanton Floer homology,
	defined in \cite{Don:YM-Floer} and denoted by $\widetilde {HF}(M)$.
\end{remark}

\section{Atiyah-Floer Conjecture and Moduli Space of Solutions to the Mixed Equation}
\label{modulispace}

In this section, we propose a program to prove Conjecture \ref{conj7878}. The main geometrical input in this program is a moduli space which is a mixture of the moduli space of Anti-Self-Dual connections and pseudo-holomorphic curves. Here we describe the version introduced in \cite{Li}. Similar moduli spaces appeared in \cite{fu4}. Analogous mixed moduli spaces are also being used by Max Lipyanskiy and the authors to prove an $\SO(3)$-analogue of the Atiyah-Floer conjecture \cite{DFL}. 

Suppose $M$ is an integral homology sphere and a Heegaard splitting as in \eqref{form73} is fixed for $M$. Therefore, we can form the symplectic manifold $X={\widehat R}(\Sigma,p,\theta)$ and the Lagrangian submanifolds $L_i=\tilde R(H^i_g,p)$. For the simplicity of exposition, we assume  that $L_0$ and $L_1$ have clean intersection. Recall that Lemma \ref{lem7474} states that in general we can perturb these Lagrangians by Hamiltonian isotopies to ensure that this assumption holds.


Let the domain $W$ in the complex plane $\C$ be given as in Figure \ref{zu1}. 
\begin{figure}[h]
\centering
\includegraphics[scale=0.3]{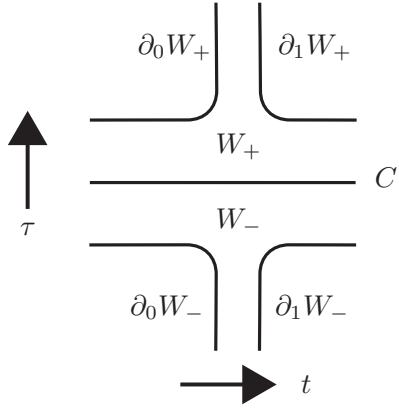}
\caption{The domain $W$}
\label{zu1}
\end{figure}
We also decompose this domain into two parts $W_-$ and $W_+$ as in the figure and let $C = W_- \cap W_+$. Using the coordinate $t,\tau$ in the figure, the line $C$ is the part $\tau = 0$. The domain $W$ has four boundary components, denoted by $\partial_0W_-$, $\partial_1W_-$, $\partial_0W_+$, $\partial_1W_+$, and four ends as below:
\begin{equation}
\aligned
&\{(t,\tau) \mid \tau \in [-1,1], t < -K_0\},
\quad
\{(t,\tau) \mid \tau \in [-1,1], t > K_0\},
\\
&\{(t,\tau) \mid t \in [-1,1], \tau < -K_0\},
\quad
\{(t,\tau) \mid t \in [-1,1], \tau > K_0\}.
\endaligned
\end{equation}
We fix a Riemannian metric $g_W$ on $W$ which coincides with the standard Riemannian metric on the complex plane where $\vert t\vert$ or $\vert \tau\vert$ is large and outside a small neighborhood of $\partial_0W_+ \cup \partial_1W_+$. We also require that the metric is isometric to $(-\epsilon,0] \times \R$ on a small neighborhood of $\partial_0W_{+}$, $\partial_1W_{+}$.

Fix a product metric on the product 4-manifold $W_+ \times \Sigma_g$. We glue $H_g^0 \times \R$ and $H_g^1  \times \R$ to the boundary components $\Sigma_g \times \partial_0W_+$
and $\Sigma_g \times \partial_1W_+$ of $W_+ \times \Sigma_g$, respectively. We will denote the resulting 4-manifold with $Y_+$ (cf. Figure \ref{zu2}).
\begin{figure}[h]
\centering
\includegraphics[scale=0.3]{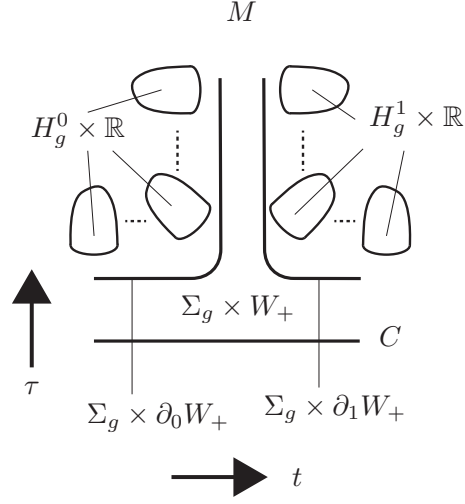}
\caption{The 4-dimensional manifold $Y_+$}
\label{zu2}
\end{figure}
The manifold $Y_+$ has three ends and one boundary component which is $\Sigma_g \times C$. The three ends correspond to the part $t \to \pm \infty$ and $\tau \to + \infty$, and they can be identified with:
\begin{equation}\label{form7171}
	H^0_g \times (-\infty,-K_0),\quad H^1_g \times (K_0,+\infty),\quad (H^0_g \cup_{\Sigma_g} H^1_g) \times (K_0,+\infty).
\end{equation}
We extend the product Riemannian metric on $\Sigma_g \times W_+$ to $Y_+$ so that the ends in \eqref{form7171} have the product Riemannian metric. Note that $H^0_g \cup_{\Sigma_g} H^1_g$ in \eqref{form7171} is the integral homology sphere $M$.

Consider the decomposition:
\[
  L_0 \cap L_1 = \{\theta\} \cup \bigcup_{a \in \mathcal A} R_a
\]
as in \eqref{decompLL}. The set $\mathcal A$ is identified with the set of irreducible flat connections on the trivial $SU(2)$-bundle over $M$.
\begin{definition}\label{defn71}
	Let $a,b \in \mathcal A$. We say the pair $(u,A)$ satisfies the {\it mixed equation}, if they satisfy the following properties.
	The first two conditions are constraints on the map $u$:
	\begin{enumerate}	
            	\item[(1.1)] $u: W_- \to X \setminus D$ is a holomorphic map with finite energy. 
            	Here $X$ and $D$ are given in \eqref{symp-div}, and the energy of $u$ is defined to be:
            	\[
            	  \int_{W_0} u^* \omega
            	\]
            	with $\omega$ being the symplectic form of $X$.
            	\item[(1.2)] The map $u$ satisfies the boundary conditions $u(\partial_0W_-)\subset L_0$ and 
            	$u(\partial_1W_-)\subset L_1$. Moreover, we require that for $t \in [-1,1]$, we have:
            	\[
            	  \lim_{\tau\to -\infty} u(\tau,t) = \frak p \in R_a.
            	\]
            	Here $\frak p$ is an element of $R_a$ which is independent of $t$.
	\end{enumerate}
	The next two conditions are on the connection $A$:
	\begin{enumerate}
            	\item[(2.1)] $A$ is a connection on the trivial $\SU(2)$-bundle over $Y_+$ which satisfies the anti-self-duality equation
            	\begin{equation*}
            		F^+(A)=0 
            	\end{equation*}
            	and its energy given by:
            	\[
            	 \int_{Y_+} \vert F_A\vert^2 \, dvol
            	\]
            	is finite.
            	\item[(2.2)]
            	For $\tau>K_0$, let $A_\tau$ denote the restriction of $A$ to $(H^0_g \cup_{\Sigma_g} H^+_g) \times \{\tau\} \cong M $.
            	The connection $A_{\tau}$ on $M$ converges to the flat connection $b$ as $\tau$ goes to $+\infty$.
	\end{enumerate}
	The last three conditions are matching conditions for $u$ and $A$ on the borderline $C$:
	\begin{enumerate}
	\item[(3.1)]
            	If $(t,0) \in C$, then $u(t,0) \in \widehat \mu^{-1}(0)$, where $\widehat \mu:X \to \frak{su}(2)$ is the moment map of
            	Theorem \ref{thm3434}. 
            	\item[(3.2)]
            	The restriction of $A$ to $\Sigma_g \times \{(t,0)\} \subset \partial X_+$, denoted by $A_{(t,0)}$, is flat for any $(t,0) \in C$.
            	\item[(3.3)]
            	The gauge equivalence class of the flat connection $A_{(t,0)}$ coincides with the equivalence class $[u(t,0)]$ 
            	of $u(t,0)$ in ${\widehat \mu}^{-1}(0)/{\rm PU}(2) = R(\Sigma)$. (See Theorem \ref{thm3434}.)
	\end{enumerate}
\end{definition}
\begin{definition}\label{defn72}
	Suppose $(u,A)$, $(u',A')$ are two pairs that satisfy the mixed equation. These two elements are equivalent, if 
	there exists a gauge transformation $g$ on $Y_+$ and $h \in {\rm PU}(2)$ such that:
	\[
	A' = g_*A\qquad u' = h u.
	\]
	We will write $\overset{\circ}{\mathcal M}(W_-,Y_+,L_0,L_1;a,b;E)$ for the space of equivalence classes of the pairs 
	$(u,A)$ satisfying the mixed equation and the following energy constraint:
	\[
	  E=\int_{W_0} u^* \omega + \int_{Y_+} \Vert F_A\Vert^2.
	\]
\end{definition}

We wish to show that the moduli space $\overset{\circ}{\mathcal M}(W_-,Y_+,L_0,L_1;a,b;E)$ behaves nicely and it can be compactified in a way that we can use it to construct an isomorphism between $I^*(M)$ and $I_{symp}^*(M)$. This requires us to generalize the analytical results of \cite{Li, DFL}. The matching condition in Definition \ref{defn72} can be regarded as a Lagrangian boundary condition associated to a Lagrangian correspondence from the infinite dimensional space of $\SU(2)$ connections over $\Sigma$ to $X$. A similar infinite dimensional Lagrangian correspondence appears in \cite{Li, DFL}. However, the Lagrangian correspondence in the present context is singular. Therefore, proving the required analytical results for the moduli space $\overset{\circ}{\mathcal M}(W_-,Y_+,L_0,L_1;a,b;E)$ (such as Fredholm theory, regularity and compactness) seems to be more challenging. Nevertheless, we conjecture that this moduli space satisfies these properties and it can be compactified to a space ${\mathcal M}(W_-,Y_+,L_0,L_1;a,b;E)$.


This compactification ${\mathcal M}(W_-,Y_+,L_0,L_1;a,b;E)$ is expected to have a virtual fundamental chain whose boundary is the union of the following two types of spaces. The first type is:
\begin{equation}\label{boundary73}
	\overline{\mathcal M}^{\rm RGW}(a,c;\beta;L_0,L_1)\times {\mathcal M}(W_-,Y_+,L_0,L_1;c,b;E_2)
\end{equation}
with $c \in \mathcal A$ and $\omega[\beta]+E_2 = E$, and the second type is:
\begin{equation}\label{boundary74}
{\mathcal M}(W_-,Y_+,L_0,L_1;a,c;E_1)
\times
\mathcal M(c,b;E_2;M)
\end{equation}
with $c \in \mathcal A$ and $E_1 + E_2 = E$. Here $\overline{\mathcal M}^{\rm RGW}(a,c;\beta;L_0,L_1)$ and $\mathcal M(c,b;E_2;M)$ are the moduli spaces that appeared in Section \ref{Sec:AFori}. 

Assuming the existence of compactification ${\mathcal M}(W_-,Y_+,L_0,L_1;a,b;E)$ with the above properties, we define a map $\Phi:C^*_{symp}(M) \to C^*(M)$ as follows:
\[
  \Phi(a) = \sum_{b,E}\#{\mathcal M}(W_-,Y_+,L_0,L_1;a,b;E)[b].
\]
where the sum is over all choices of $E$ and $b$ such that ${\mathcal M}(W_-,Y_+,L_0,L_1;a,b;E)$ is 0-dimensional. The signed number of points in this 0-dimensional moduli space is denoted by $\#{\mathcal M}(W_-,Y_+,L_0,L_1;a,b;E)$.
By a standard argument applying to the 1-dimensional moduli spaces ${\mathcal M}(W_-,Y_+,L_0,L_1;a,b;E)$ and using the description of the boundary of this moduli space in \eqref{boundary73}, \eqref{boundary74}, we can conclude this implies that $\Phi$ is a chain map. The energy $0$ part of the moduli space ${\mathcal M}(W_-,Y_+,L_0,L_1;a,b;0)$ is empty if $a \ne b$ and has one point if $a=b$. It implies that $\Phi$ induces an isomorphism between corresponding Floer homologies.

\begin{remark}
	Note that in Definition \ref{defn71}, we do not assume any particular asymptotic boundary conditions on the ends where $t \to \pm\infty$.
	In fact, the finiteness of the energy should imply that the pair $(u,A)$ converges to a constant map to a flat connection on $H_g^i$ as $t \to \pm\infty$.
	Therefore, the choices of Definition \ref{defn71} imply asymptotic convergence to the fundamental class of ${\tilde R(H^0_g,p)}$ on these ends.
	This particular choice of the asymptotic boundary condition at $t \to \pm\infty$ is very important to show that $\Phi$ induces an isomorphism in homology. 
	In fact, we use it to show that the contribution of the lowest energy part to $\Phi$ is the identity map.
\end{remark}

\begin{remark}
	The map $\Phi$ is defined similar to some chain maps which appear in \cite{fu3}. In the definition of these chain maps the Lagrangian ${\tilde R(H^0_g,p)}$ is replaced with arbitrary Lagrangian submanifold of 
	the underlying symplectic manifold. However, the idea that such maps can be used to construct isomorphisms is inspired by Lekili and  Lipyanskiy's work in \cite{LL}, where the methods of \cite{fu3} is revived
	 in a similar context.
\end{remark}

\begin{remark}
	The special case of the $\SO(3)$-Atiyah-Floer conjecture for mapping tori of surface diffeomorphisms was proved
	in the seminal work of Dostoglou and Salamon \cite{DS}. Their proof uses an adiabatic limit argument and is based on the following crucial observation.
	Consider the 4-manifold $\Sigma_g \times W$, where $W$ is a surface, and let the metric on $\Sigma_g$ degenerate.
	Then the ASD equation turns into the holomorphic curve equation 
	from $W$ to the space $R(\Sigma_g)$ of flat connections on $\Sigma_g$.
	Later, Salamon proposed a program for the original version of the Atiyah-Floer conjecture using similar adiabatic limit argument \cite{Sa} and this 
	approach was pursued further by Salamon and Wehrheim \cite{Sawe, We, We2}. 
	The extension of the adiabatic limit argument to the general case of the $\SO(3)$-analogue of the Atiyah-Floer conjecture is also being investigated 
	by David Duncan \cite{Duncan}. 
	The adiabatic limit argument has the potential advantage of finding a relationship between the moduli spaces involved in gauge theory and symplectic geometry, and not only a relationship at the level of Floer homologies.
	The drawback is one has to face complicated analytical arguments.
	We believe the approach discussed in this section (and the corresponding one in \cite{DFL} for the 
	$\SO(3)$-analogue of the Atiyah-Floer conjecture) has less analytical difficulties because it uses the functorial properties 
	of Floer homologies. A similar phenomenon appears in the proof of the connected sum theorem for instanton Floer homology 
	of integral homology spheres where the ``functorial'' approach \cite{fu15,Don:YM-Floer} seems to be easier than the adiabatic limit argument \cite{WL}.
\end{remark}

\section{Yang-Mills Gauge Theory and 3-Manifolds with Boundary}
\label{Sec:conjectures}


In Section \ref{Sec:AFori}, we sketched the construction of $\overline I_{symp}(M)$, as a module over $H^*_{{\rm PU}(2)}$, for an integral homology sphere $M$. This invariant is defined by considering Yang-Mills gauge theory on principal $\SU(2)$-bundles. It is natural to ask to what extent this construction can be generalized to arbitrary 3-manifolds and arbitrary choice of principal bundles. In the following conjecture, suppose $G$ is given as in Section \ref{sec:extendedmoduli}. 
\begin{conjecture}
	Suppose $\mathcal E$ is a $G$-bundle over a 3-manifold $M$. Then there is a $H^*_{G}$-module $\overline I^*_{symp}(M,\mathcal E)$
	which is an invariant of the pair $(M,\mathcal E)$. In the case that $M$ is an integral homology sphere and $G={\rm PU}(2)$, 
	this invariant matches with the construction of Section \ref{Sec:AFori}.
\end{conjecture}
\noindent 
To be more precise, we expect that the above invariant is defined using Lagrangian Floer homology on appropriate moduli spaces of flat $G$-connections over Riemann surfaces. We shall propose a plan for the construction of this invariant in Section \ref{sec:generalgroup}. The tools from symplectic topology required for this construction are discussed in the next section.

In another level of generalization, one can hope to define symplectic instanton Floer homology for 3-manifolds with boundary. In order to state the expected structure of symplectic instanton Floer homology for 3-manifolds with boundary, we need to recall the definition of $A_\infty$-categories:

\begin{definition}
	Fix a commutative ring $R$. An $A_{\infty}$-category $\mathscr C$ consists of a set of objects $\mathscr{OB}(\mathscr C)$, 
	a graded $R$-module of {\it morphisms} $\mathscr C(c,c')$ for each pair of objects $c,c' \in \mathscr{OB}(\mathscr C)$, 
	and the structural operations $\frak m_k : \bigotimes_{i=1}^k \mathscr C(c_{i-1},c_i)\to \mathscr C(c_{0},c_k)$ of degree $k-2$
	for each $k\geq 1$. 
	The multiplication maps $\frak m_k$ are required to satisfy the following relations:
	\begin{equation}\label{formula25}
		\sum_{k_1+k_2=k+1}\sum_{i=0}^{k_1-1}(-1)^* \frak m_{k_1}(x_1,\dots,x_i,\frak m_{k_2}(x_{i+1},\dots,x_{k_2}),\dots,x_k)=0
	\end{equation}
	where $* = i +\sum_{j=1}^i \deg x_j$.
\end{definition}

Let $M$ be a 3-dimensional manifold whose boundary is decomposed as below:
\begin{equation*}
	\partial M = -\Sigma_1 \sqcup  \Sigma_2.
\end{equation*}
where $-\Sigma_1$ denotes the 3-manifold $\Sigma_1$ with the reverse orientation. Suppose also $\mathcal E$ is a $G$-bundle on $M$ whose restriction to $\Sigma_i$ is denoted by $\mathcal F_i$. We shall say $(M,\mathcal E)$ is a {\it cobordism} from $(\Sigma_1,\mathcal F_1) $ to $(\Sigma_2,\mathcal F_2)$ and we shall write:
\begin{equation}\label{3mfdassum}
	(M,\mathcal E): (\Sigma_1,\mathcal F_1) \to(\Sigma_2,\mathcal F_2).
\end{equation}

\begin{conjecture}\label{conj21}
            \begin{enumerate}
                    \item[(A-1)]
                    	For any $G$-bundle $\mathcal F$ over a Riemann surface, 
                    	there is a unital\footnote{A filtered $A_{\infty}$ category is unital if it has a strict unit.} filtered $A_{\infty}$-category 
                    	$\fI(\Sigma,\mathcal F)$ over the ring $\Lambda_0^{H^*_{G}}$ The $A_\infty$-category associated to $(-\Sigma,\mathcal F) $ is 
			$\fI(\Sigma,\mathcal F)^{\rm op}$, the opposite $A_{\infty}$ category of 
			$\fI(\Sigma,\mathcal F)$\footnote{The opposite $A_{\infty}$ category is defined 
			by reversing the direction of arrows. See \cite[Definition 7.8]{fu6}.}.
			Moreover, if $\Sigma$ is the disjoint union of two surfaces $\Sigma_1$, $\Sigma_2$, and 
			the restriction of $\mathcal E$ to $\Sigma_i$ is $\mathcal E_i$, then we have the identification:
			\[
			  \fI(\Sigma,\mathcal F)\cong \fI(\Sigma_1,\mathcal F_1)\otimes \fI(\Sigma_2,\mathcal F_2).
			\]
			Here the right hand side is the tensor product of filtered $A_{\infty}$-categories. (See \cite{Limo,FuFu50}.)
                    \item[(A-2)]
			For any pair as in \eqref{3mfdassum}, there is a filtered 
			$A_{\infty}$ functor\footnote{In the terminology of \cite{FuFu50}, 
			$\fI_{(M,\mathcal E)}$ is a {\it strict} filtered $A_{\infty}$ functor.}:
			\[ 
			  \fI_{(M,\mathcal E)} : \fI(\Sigma_1,\mathcal F_1)\to 
			  \fI(\Sigma_2,\mathcal F_2).
			\]
			The $A_\infty$-functor associated  to $(-M,\mathcal E)$ is the adjoint functor of 
			$\fI(M,\mathcal E)$\footnote{See \cite{FuFu50} for the definition of adjoint functor of a filtered $A_{\infty}$functor.}.
            	   \item[(A-3)]
			For $i=1,2$, let $(M_i,\mathcal E_i)$ be a 3-dimensional cobordism from $(\Sigma_i,\mathcal F_i)$ to 
			$(\Sigma_{i+1},\mathcal F_{i+1})$. 
			Let $(M,\mathcal E)$ be the result of composing these cobordisms along $(\Sigma_2,\mathcal F_2)$. Then:
			\begin{equation}\label{form14}
				\fI_{(M,\mathcal E)}\cong \fI_{(M_2,\mathcal E_2)}\circ \fI_{(M_1,\mathcal E_1)}.
			\end{equation}
			Here $\circ$ is the composition of filtered $A_{\infty}$-functors 
			and $\cong$ is the homotopy equivalence of filtered 
			$A_{\infty}$ functors\footnote{Two filtered $A_{\infty}$ functors are  homotopy equivalent if they are 
			homotopy equivalent in the functor category. (See \cite[Theorem 7.55]{fu6}.)}
			from $\fI(\Sigma_1,\mathcal F_1)$ to $\fI(\Sigma_2,\mathcal F_2)$.
	\end{enumerate}
\end{conjecture}
The following conjecture extends Conjecture \ref{conj21} to the case that at least one of the ends of $(M,\mathcal E)$ is empty:

\begin{conjecture}
	Let $(M,\mathcal E)$ be as in \eqref{3mfdassum}:
	\begin{enumerate}
		\item[(B-1)]
			If $\Sigma_1=\emptyset$, then $\fI_{(M,\mathcal E)}$ is an object of $\fI(\Sigma_2,\mathcal F_2)$.
		\item[(B-2)]
			If $\Sigma_2=\emptyset$, then $\fI_{(M,\mathcal E)}$ is a filtered $A_{\infty}$ functor from $\fI(\Sigma_1,\mathcal F_1)$ to 
			$\mathcal{CH}$, where $\mathcal{CH}$ is the $DG$ category of chain complexes.
		\item[(B-3)]
			If $\Sigma_1=\Sigma_2=\emptyset$, then $\fI_{(M,\mathcal E)}$ is a chain complex over $\Lambda^{H^*_{G}}_0$.
	\end{enumerate}
\end{conjecture}
The next conjecture is an extension of part (A-3) of Conjecture \ref{conj21} to the case that one of the boundary components is empty:
\begin{conjecture}
	Let $(M_1,\mathcal E_1)$ and $(M_2,\mathcal E_2)$ be as in part {\rm (A-3)} of Conjecture \ref{conj21}:
	\begin{enumerate}
		\item[(C-1)] If $\Sigma_1 = \emptyset$, then:
				\begin{equation}\label{eqC1}
					\fI_{(M,\mathcal E)}\cong \fI_{(M_2,\mathcal E_2)}(\fI_{(M_1,\mathcal E_1)})
				\end{equation}
				This is a homotopy equivalence of objects in the category $\fI(\Sigma_3,\mathcal F_3)$.
		\item[(C-2)] If $\Sigma_2= \emptyset$, then \eqref{form14} as the homotopy equivalence of 
			$A_\infty$ functors from the category $\fI(\Sigma_1,\mathcal F_1)$ to the category $\mathcal{CH}$ holds.
		\item[(C-3)] If $\Sigma_1 =\Sigma_2= \emptyset$, then \eqref{eqC1} as a chain homotopy equivalence between chain complexes
		holds.
	\end{enumerate}
\end{conjecture}


\section{Lagrangians and $A_\infty$-categories}
\label{Sec:ainfinity}


For a given symplectic manifold $(X,\omega)$, we can define an $A_{\infty}$-category over the universal Novikov ring $\Lambda_0^\R$, which is usually denoted by $\frak{Fuk}(X,\omega,\fL)$. The objects of this category $\frak{Fuk}(X,\omega,\fL)$ are defined using immersed Lagrangian submanifolds, and $\fL$ denotes a collection of such Lagrangians.


Suppose $\widetilde L$ is an immersed Lagrangian submanifold of $X$ given by $\iota:\widetilde L \to X$ where the self-intersections of $X$ are transversal. Define $CF(\widetilde L,\widetilde L)$ to be the cohomology group $H^*(\widetilde L \times_X \widetilde L,\Lambda^\R)$ where $\widetilde L \times_X \widetilde L$ is the fiber product of the map $\iota$ with itself\footnote{To be more precise, one needs to start with a chain model for this cohomology group. As it is shown in \cite{fooobook}, this chain model can be replaced with the cohomology groups by an algebraic argument.}. Therefore, $CF(\widetilde L,\widetilde L)$ is the direct sum of $H^*(\widetilde L)$ and a free abelian group generated by the self-intersection points of $\widetilde L$. For $\beta \in H_2(X,\iota(\widetilde L);\Z)$, let $\mathcal M_{k+1}(\widetilde L;\beta)$ be the compactified moduli space of pseudo-holomorphic disks with $k+1$ boundary marked points. The elements of $\mathcal M_{k+1}(\widetilde L;\beta)$ are required to represent the homology class $\beta$ and need to satisfy the Lagrangian boundary condition. The boundary of an element of $\mathcal M_{k+1}(\widetilde L;\beta)$ has to be mapped to $\iota (\widetilde L)$, and away from the marked points, it can be lifted to $\widetilde L$. (See Figure \ref{zu2} for a schematic picture and \cite[Definition 2.1.27]{fooobook} and \cite[Section 4]{AJ} for the precise definitions of these moduli spaces.) The moduli space $\mathcal M_{k+1}(\widetilde L;\beta)$ can be used to form the following diagram:
\begin{equation}
\xymatrix{
(\widetilde L\times_X\widetilde L)^k&&&\mathcal M_{k+1}(\widetilde L;\beta)  \ar[lll]_{({\rm ev}_1,\dots,
{\rm ev}_k)\hspace{.4cm}}\ar[rrr]^{\hspace{.8cm}{\rm ev}_0} &&&(\widetilde L\times_X\widetilde L)
}
\nonumber
\end{equation}
where ${\rm ev}_i$ for $0\leq i\leq k$, is the evaluation map at the $i^{\rm th}$ marked point. A standard `pull-up-push-down construction' applied to these diagrams for various choices of $\beta$ determines a map $\fm_k:CF(\widetilde L,\widetilde L)^{\otimes k} \to CF(\widetilde L,\widetilde L)$ for any $k\geq 0$.
\begin{figure}[h]
\centering
\includegraphics[scale=0.3]{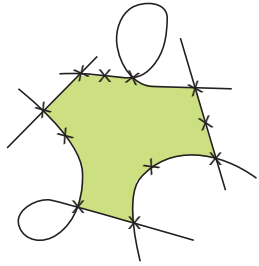}
\caption{The operation $\frak m_k$}
\label{zu2}
\end{figure}

Next, let $\fL$ be a finite family of immersed spin Lagrangian submanifolds of $X$. We say that this family is clean if for any two elements $(\widetilde L_1, \iota_1)$ and $(\widetilde L_2, \iota_2)$ of this family the fiber product $\widetilde L_1\times_X \widetilde L_2$ is a smooth manifold and the tangent space at each point of $\widetilde L_1\times_X \widetilde L_2$ is given by the fiber product of the tangent spaces of $\widetilde L_1$ and $\widetilde L_1$. 
Here we include the case $\tilde L_1 =\tilde L_2$.
For any two such elements of $\fL$, we define $CF(\widetilde L_1,\widetilde L_2)$ to be the cohomology group $H^*(\widetilde L_1 \times_X \widetilde L_2)$. Then the construction of the previous paragraph can be modified to define the following maps for any $k \geq 0$ and for any sequence $\{(\widetilde L_i, \iota_i)\}_{0\leq i \leq k}$ of the elements of $\fL$ \cite{fooobook,AJ}:
\[
  \fm_k : \bigotimes_{i=1}^k CF(\widetilde L_{i-1},\widetilde L_i) \to CF(\widetilde L_{0},\widetilde L_k) 
\]  

The maps $\fm_k$ satisfy the analogues of the $A_\infty$-relations in \eqref{formula25}. However, the map $\fm_0$ does not have to vanish in general\footnote{An $A_\infty$-category with a non-vanishing $\fm_0$ is called a {\it curved $A_\infty$-category.}}. Therefore, these maps cannot immediately be used to define an $A_\infty$-category. This issue can be fixed with the aid of {\it bounding cochains}. For an immersed Lagrangian $(\widetilde L,\iota)$, an element $b\in H^{odd}(\widetilde L,\Lambda_0)$ is called a bounding cochain if it is divisible by $T^\epsilon$ for a positive $\epsilon$, and it satisfies the following {\it Maurer-Cartan} equation:
\begin{equation}\label{formMC}
	\sum_{k=0}^{\infty} \fm_{k}(b,\dots,b) = 0,
\end{equation}

An object of the category $\frak{Fuk}(X,\omega,\fL)$ is a pair $(\widetilde L,b)$ where $\widetilde L$ is an element of $\fL$ and $b$ is a bounding cochain. The module of morphisms for two objects $(\widetilde L_0,b_0)$ and $(\widetilde L_1,b_1)$ is defined to be $CF(\widetilde L_0, \widetilde L_1)$. The structural map $\fm_k^{\vec b} : \bigotimes_{i=1}^k CF(\widetilde L_{i-1},\widetilde L_i) \to CF(\widetilde L_{0},\widetilde L_k)$ for a sequence of objects $\{(\widetilde L_i, \iota_i,b_i)\}_{0\leq i \leq k}$ is also defined as follows:
\begin{equation}\label{modified-m}
	\fm_k^{\vec b}(p_1,\dots,p_k):=\sum_{l_0\geq, \dots, l_k\geq 0}\fm_{k+l_0+\dots+l_k}(b_0^{\otimes l_0},p_1,b_1^{\otimes l_1},\dots,
	b_{k-1}^{\otimes l_{k-1}},p_k,b_{k}^{\otimes l_{k}})
\end{equation}
Using the results of \cite{fooobook,fooobook2}, it is shown in \cite{anchor,AFOOO,FuFu50} that $\frak{Fuk}(X,\omega,\fL)$ is an $A_\infty$-category in the case that $\fL$ consists of only embedded Lagrangians. The more general case of immersed Lagrangians is treated in \cite{AJ}.

Suppose $L_0$ and $L_1$ are two monotone and embedded Lagrangians in $X$ that satisfy the condition (m.a) of Section \ref{sec:equivariant}. Then  the map $\fm_0:\Lambda_0 \to CF(L_i,L_i)$ vanishes and we can associate the trivial bounding cochain to each of these Lagrangians. The map $\fm_1:CF(L_0,L_1) \to CF(L_0,L_1)$ defines a differential. The homology of this chain complex is the same as Oh's Lagrangian Floer homology for monotone Lagrangians \cite{Oh}.

We can also consider equivariant version of the category $\frak{Fuk}(X,\omega,\fL)$. The following theorem provides the main ingredient for the equivariant construction:

\begin{theorem}\label{thm5151}
	 Let $G$ be a Lie group acting on $(X,\omega)$. Let $\fL$ be a clean collection of immersed Lagrangians which are equivariant with respect to the action of $G$. 
	For any sequence $\{(\widetilde L_i, \iota_i)\}_{0\leq i \leq k}$ of the elements of $\fL$, 
	there exists a $H_G^*$-linear homomorphism:
	\[
	  \fm_k^G: \bigotimes_{i=1}^k H^*_G(\widetilde L_{i-1}\times_X \widetilde L_i,\Lambda_0) \to 
	  H^*_G(\widetilde L_{0}\times_X \widetilde L_k,\Lambda_0)
	\]
	which satisfies the $A_{\infty}$-relations in \eqref{formula25}. Moreover, the following diagram commutes:
	\begin{equation}
		\begin{CD}
	  	  	 \bigotimes_{i=1}^k H^*_G(\widetilde L_{i-1}\times_X \widetilde L_i,\Lambda_0) @ >\frak m_k^G>>
		   	  H^*_G(\widetilde L_{0}\times_X \widetilde L_k,\Lambda_0)\\
			@ V{}VV @ VV{}V\\
			\bigotimes_{i=1}^k H^*(\widetilde L_{i-1}\times_X \widetilde L_i,\Lambda_0) @ >\frak m_k>> 
			H^*(\widetilde L_{0}\times_X \widetilde L_k,\Lambda_0).
		\end{CD}
	\end{equation}
	Here the vertical arrows are canonical maps from equivariant cohomology to de-Rham cohomology. 
\end{theorem}

\begin{proof}[Sketch of the proof]
        For the simplicity of exposition, assume that the immersed Lagrangians are the same as each other. 
        We use the equivariant Kuranishi structure on the space 
        $\mathcal M_{k+1}(L;\beta)$ \cite{FuFu5} and an approximation of the universal principal $G$-bundle $EG \to BG$ to obtain:
        \begin{equation}
                \xymatrix{
                &&\mathcal M_{k+1}(L;\beta) \times_G EG(N)  \ar[dl]^{({\rm ev}_1,\dots,
                {\rm ev}_k)}\ar[dr]_{{\rm ev}_0}   \\
                &(\widetilde L\times_X\widetilde L)^k \times_G EG(N) \ar[d]
                && 
                L \times_G EG(N) \\
                &((\widetilde L\times_X\widetilde L) \times_G EG(N))^k
                }
                \nonumber
        \end{equation}
        where ${\rm ev}_i$ for $0\leq i\leq k$, is the evaluation map at the $i$-th marked point. 
        By a formula similar to (\ref{form31100}), we can define operations:
        \[
          \frak m^{G,N}_{k,\beta} : H(L \times_G EG(N))^{\otimes k} \to H(L \times_G EG(N)). 
        \]
        Taking the limit $N\to \infty$, we obtain the operation $\frak m^{G}_{k,\beta}$ between the equivariant cohomology groups.
        Then
        $\frak m^G_k = \sum_{\beta} T^{\omega(\beta)}\frak m^{G}_{k,\beta}$ is the required $A_{\infty}$ operation.
\end{proof}


\begin{definition}
	An element $b \in H^{{\rm odd}}_G(L;\Lambda_0)$ is a {\it $G$-equivariant bounding cochain},
	 if $b$ is divisible by $T^{\epsilon}$ for a positive $\epsilon$ and $b$ satisfies (\ref{formMC}), where
	 $\frak m_k$ is replaced with $\frak m_k^G$.
\end{definition}

The following theorem claims the existence of the $G$-equivariant analogue of $\frak{Fuk}(X,\omega,\fL)$. The geometric content of this theorem is given in Theorem \ref{thm5151}:
\begin{theorem}\label{thm54}
	Let $\fL$ be a clean collection of $G$-equivariant immersed Lagrangian submanifolds of $(X,\omega)$. 
	There exists a (filtered) $A_{\infty}$-category $\frak{Fuk}^G(X,\omega,\fL)$ whose objects are pairs of the form
	$(L,b)$ where $L \in \fL$ and $b$ is a $G$-equivariant bounding cochain. The structural maps $\fm_k^{\vec b,G}$
	are also defined by applying the analogue of the formula of \eqref{modified-m} to the maps $\fm_k^G$.
\end{theorem}

The above theorem can be regarded as the generalization of the results of Section \ref{sec:equivariant} on $G$-equivariant Lagrangian Floer homology. Similarly, the techniques of \cite{DF} can be used to extend the results of Section \ref{sec:divisor complement}. More generally, we 
can also consider $A_\infty$-category associated to $G$-invariant Lagrangians in the complement of a smooth divisor:

\begin{theorem}\label{cat-div-comp}
	Let $(X,\omega)$ be a symplectic manifold with a Lie group $G$ acting on $X$ by symplectomorphisms. 
	Let $D$ be a $G$-invariant smooth divisor in $X$ such that $X$ admits a K\"ahler 
	structure in a neighborhood of $D$ compatible with the symplectic form 
	$\omega$. Let $\fL$ be a clean collection of $G$-equivariant immersed Lagrangian submanifolds of $(X,\omega)$. 
	\begin{enumerate}
		\item[(i)] There are operations:
			\begin{equation}\label{div-equiv-Ainfty}
			    \fm_k^G : \bigotimes_{i=1}^k CF(\widetilde L_{i-1},\widetilde L_i) \to CF(\widetilde L_{0},\widetilde L_k) 
			\end{equation}
			for any sequence $\{(\widetilde L_{i},\iota_i)\}_{0\leq i \leq k}$ of the elements of $\fL$. 
			These operations satisfy $A_\infty$-relations.
		\item[(ii)] There exists a (filtered) $A_{\infty}$-category $\frak{Fuk}^G(X\backslash D,\omega,\fL)$ whose objects are pairs of the form
			$(L,b)$ where $L \in \fL$ and $b$ is a bounding cochain in $CF(\widetilde L,\widetilde L)$ with respect to the operators defined in
			Item {\rm (i)}. The structural operations of $\frak{Fuk}^G(X\backslash D,\omega,\fL)$ are given by modifications of 
			the operators in \eqref{div-equiv-Ainfty} as in \eqref{modified-m}.
	\end{enumerate}
\end{theorem}

\begin{conjecture}\label{cong59}
	Theorems \ref{mainthm} and \ref{cat-div-comp} still hold in the case that $D$ is a normal crossing divisor with respect to a K\"ahler structure in
	a neighborhood of $D$ which is compatible with $\omega$.
\end{conjecture}

\begin{remark} \label{normal-crossing}
	As in Section \ref{sec:divisor complement}, we need to use a non-standard compactification of 
	the moduli spaces of pseudo-holomorphic disks in $X \setminus D$ to prove Theorem \ref{cat-div-comp}. 
	It is plausible that the compactification appearing in relative Gromov-Witten theory for the complements of normal crossing divisors \cite{groSie}
	can be employed to prove Conjecture \ref{cong59}. 	The analysis of Gromov-Witten theory for the complement of normal crossing
	divisors is much more subtle than the case of the complements of smooth divisors, and we would expect that a similar phenomena happens 
	in the construction of Lagrangian Floer homology and the category $\frak {Fuk}(X\backslash D,\omega)$.
\end{remark}


\newtheorem{conds}[theorem]{Condition}

\section{Cut-down Extended Moduli Spaces for Other Lie Groups}
\label{sec:generalgroup}

In this section, we discuss a plan to address the conjectures in Section \ref{Sec:conjectures}. Fix a Lie group $G$ as in Section \ref{sec:extendedmoduli} and let $T$ be a maximal torus of $G$ whose Lie algebra is denoted by $\ft$. The Lie group $G$ acts on $\frak g$ and the quotient space can be identified with the quotient $\ft/W$ of $\frak t$ by the Weyl group $W$. Let $\ft^+\subset \ft$ be a Weyl chamber of $G$. Then $\ft^+$ is a fundamental domain for the action of the Weyl group on $\frak t$, i.e., we can identify $\ft/W$ with $\ft^+$. The quotient map from $\fg$ to $\ft^+$ is denoted by $Q$. We will also write $\ft_\Z$ for the integer lattice in $\ft$. Thus $T$ is equal to the quotient $\ft/\ft_\Z$. The dual lattice of $\ft_\Z$ is denoted by $\ft_{\Z}^*$. The action of the Weyl group on $\ft$ induces actions of this group on the lattices $\ft_\Z$ and $\ft_{\Z}^*$. For a finite subset $\mathcal N=\{\alpha_1,\dots,\alpha_n\}$ of $\frak t_{\Z}^*$, we define:
\begin{equation} \label{Delta}
	\overset{\circ}\Delta_{\mathcal N}(\epsilon) = \{\xi \in \frak t \mid \forall \alpha \in \mathcal N, \alpha(\xi) < \epsilon \}
\end{equation}
Let $\Delta_{\mathcal N}(\epsilon)$ be the closure of $\overset{\circ}\Delta_{\mathcal N}(\epsilon)$. The intersection of these open and closed polytopes with the Weyl chamber $\ft^+$ is denoted by $\Delta_{\mathcal N}^+(\epsilon)$ and $\overset{\circ\hspace{3mm}}{\Delta_{\mathcal N}^+}(\epsilon)$. 
\begin{conds}\label{cond92} 
	The set $\mathcal N$ is required to satisfy the following conditions:
	\begin{enumerate}
		\item $\mathcal N$ is invariant with respect to the action of the Weyl group.
		\item $\Delta_{\mathcal N}(\epsilon)$ is compact.
		\item If $\alpha \in \mathcal N$, $\xi \in \Delta_{\mathcal N}(\epsilon)$,  $\sigma \in W$
			satisfy $\alpha(\xi) = \epsilon$ and $\sigma(\xi) = \xi$, then $\sigma(\alpha) = \alpha$.
		\item For any vertex $v$ of $\Delta_{\mathcal N}(\epsilon)$ the set 
			$\{\alpha \in \mathcal N \mid \alpha(v) = \epsilon\}$ is a $\Z$ basis of $\frak t_{\Z}^*$.
\end{enumerate}
\end{conds}

\begin{example}
	For $G = {\rm PU}(3)$, we can assume that $\ft$ is the set of diagonal matrices with diagonal entries 
	$(2\pi\bi\theta_0,2\pi\bi\theta_1,2\pi\bi\theta_2)$ with $\sum \theta_i = 0$.
	Suppose $\alpha_i\in \ft_\Z$ is the map that assigns $\theta_i$ to a diagonal matrix of this form.
	We may take $\mathcal N$ to be the set that consists of $\theta_i$ and $-\theta_i$ for $i=1,2,3$. 
	The Wely chamber and the set $\Delta_{\mathcal N}(\epsilon)$ is demonstrated in Figure \ref{zu3}.
	\begin{figure}[h]
		\centering
		\includegraphics[scale=0.3]{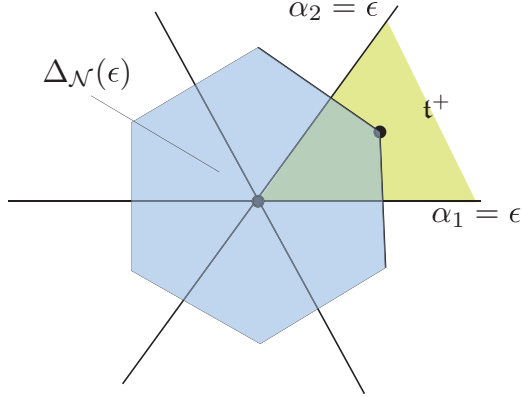}
		\caption{Wely chamber and symplectic cut}
		\label{zu3}
	\end{figure}
\end{example}

\begin{remark}
        A set of similar conditions for polytopes in $\ft^+$ are introduced by Woodward in \cite{Wo20}.
        For example, Condition \ref{cond92} (2) implies the requirements of \cite[Definition 1.1]{Wo20}. 
        Condition \ref{cond92} (4) also asserts that $\Delta_{\mathcal N}^+(\epsilon)$ is Delzant. (See \cite[page 5]{Wo20}.)
        However, our requirements are more restrictive. For example, we only consider polytopes in $\ft^+$
        which contain the origin. In fact, our definition is essentially the same as outward-positive condition in \cite{MT}.
\end{remark}

For the following proposition, let $\widehat{R}(\Sigma,\mathcal F,p)$ and $\mu : \widehat{R}(\Sigma,\mathcal F,p) \to \frak g$ be given as in \eqref{form2424}. This proposition is a consequence of well-established results on non-abelian symplectic cut \cite{Wo20,Mei,MT}:

\begin{prop}\label{prop93}
	Let $\epsilon$ be a positive real number. If $\epsilon$ is small enough, then
	there exists a compact symplectic manifold $\widehat{R}(\Sigma,\mathcal N,\mathcal F,p;\epsilon)$ 
	with a Hamiltonian $G$ action and a moment map $\widehat \mu:\widehat{R}(\Sigma,\mathcal N,\mathcal F,p;\epsilon) \to \fg$, 
	which satisfies the following properties:
	\begin{enumerate}
		\item[(i)] The image of the map $Q \circ \widehat \mu$ is equal to $\Delta^+_{\mathcal N}(\epsilon)$.
		\item[(ii)] The open subspaces $(Q\circ \mu)^{-1}(\overset{\circ\hspace{3mm}}{\Delta_{\mathcal N}^+}(\epsilon))$ and 
				$(Q\circ \widehat \mu)^{-1}(\overset{\circ\hspace{3mm}}{\Delta_{\mathcal N}^+}(\epsilon))$ are symplecomorphic.
	\end{enumerate}
\end{prop}	
	
\begin{proof}[Sketch of the proof]
	Let $\widehat{R}(\Sigma,\mathcal N,\mathcal F,p;\epsilon)_0$ denote the subspace 
	$(Q\circ \mu)^{-1}(\overset{\circ\hspace{3mm}}{\Delta_{\mathcal N}^+}(\epsilon))$ of $\widehat{R}(\Sigma,\mathcal N,\mathcal F,p)$.
	According to Proposition \ref{symp-ext}, this space has a symplectic structure if $\epsilon$ is small enough. 
	We compactify this space into a closed symplectic manifold in the following way. 
	Given $x \in \mu^{-1}(\Delta_{\mathcal N}(\epsilon))$, we may assume without loss of generality that there is $k$ such that:
	\[
	  \alpha_i(\mu(x)) = \epsilon
	\]
	if and only if $1\leq i\leq k$.
	The map $\Phi_i:\mu^{-1}(\ft)\to \R$ defined as:
	\[
	  \Phi_i(y)=\alpha_i(\mu(y))
	\]
	can be extended as a smooth function to a neighborhood $U_x\subset \widehat{R}(\Sigma,\mathcal F,p)$ of $x$. We define 
	a real-valued function $f_i$ on $U_x \times \C^k$ as follows:
	\[
	  f_i(y,\xi_1,\dots,\xi_k)  =  \epsilon^2 - \phi_i(y) - \vert \xi_i\vert^2. 
	\]
	Then the function $f:=(f_1,\dots,f_k):U_x \times \C^k\to \R^k$ is the moment map for a Hamiltonian action of $T^k$ on $U_x\times \C^k$.
	The symplectic quotient $f^{-1}(0)/T^k$ contains a dense subset which is symplectomorphic to 
	$\widehat{R}(\Sigma,\mathcal N,\mathcal F,p;\epsilon)_0\cap U_x$. To be a bit more detailed, if we map 
	$y \in \widehat{R}(\Sigma,\mathcal N,\mathcal F,p;\epsilon)_0 \cap U_x$ to the equivalence class of 
	$(y,\xi_1\,\dots,\xi_k)$ where $\xi_i = \sqrt{\vert\epsilon^2 - \Phi_i(y)\vert}$, then we obtain an open embedding of 
	$\widehat{R}(\Sigma,\mathcal N,\mathcal F,p;\epsilon)_0 \cap U_x$ into $f^{-1}(0)/T^k$.
	Condition \ref{cond92} (4) shows that $\overline U_x:=f^{-1}(0)/T^k$ is a smooth manifold. (See \cite[Proposition 6.2]{Wo20}.)
	We can glue $\overline U_x$ for various choices of $x \in \mu^{-1}(\Delta(\epsilon))$ to obtain the desired symplectic manifold 
	$\widehat{R}(\Sigma,\mathcal N,\mathcal F,p;\epsilon)$.
\end{proof}

Let $D$ denote the complement of $\widehat{R}(\Sigma,\mathcal N,\mathcal F,p;\epsilon)_0 $ in $X:=\widehat{R}(\Sigma,\mathcal N,\mathcal F,p;\epsilon)$. We expect that $\widehat{R}(\Sigma,\mathcal N,\mathcal F,p;\epsilon)$ admits a K\"ahler structure compatible with the symplectic structure of $X$ denoted by $\omega$ such that $D$ forms a normal crossing divisor in this neighborhood. Therefore, in the light of Conjecture \ref{cong59}, we make the following conjecture:

\begin{conjecture} 
	There is an $A_\infty$-category ${\frak Fuk}^G(X\backslash D,\omega)$ associated to $(X,\omega)$ and $D$ as above, where the objects of 
	this category form a family of $G$-equivariant immersed Lagrangian submanifolds of $X\backslash D$, and the morphisms of 
	this category are constructed by holomorphic maps to $X$. The homotopy equivalence of this category is independent of 
	$\mathcal N$ and $\epsilon$.
\end{conjecture}

\begin{conjecture}
	Suppose $\Sigma$ is a connected Riemann surface. Then the $A_\infty$-category has the properties of the 
	category $\fI(\Sigma,\mathcal F)$ in Conjecture \ref{conj21}.
	For a disconnected $\Sigma$, we can take the tensor product of categories associated to the connected components.
\end{conjecture}

To elaborate on this proposal, let $(M,\mathcal E)$ be a cobordism from the empty pair to $(\Sigma,\mathcal F)$. As in the case of handlebodies discussed in Section \ref{sec:extendedmoduli}, we can associate to $(M,\mathcal E)$ a subspace of $\widehat{R}(\Sigma,\mathcal F,p)$ which lives in $\mu^{-1}(0)$. Therefore, it can be also regarded as a subspace of $X=\widehat{R}(\Sigma,\mathcal N,\mathcal F,p;\epsilon)$. A holonomy perturbation of this subspace can be used to turn this space into an immersed Lagrangian submanifold $\widetilde L_{(M,\mathcal E)}$. 
\begin{conjecture}\label{conj84}
	There exists a bounding cochain $b_M$ in
	\[
	  H_G(\widetilde L_{(M,\mathcal E)}  \times_{X} \widetilde L_{(M,\mathcal E)} ).
	\]
	Together with $b_M$, the immersed Lagrangian submanifold $\widetilde L_{(M,\mathcal E)}$ determines an object of 
	${\frak Fuk}^G(X\backslash D,\omega)$.
\end{conjecture}
If the construction of various moduli space can be carried out as expected, then the proof of Conjecture \ref{conj84} is expected to be similar to the similar result in the case of $\SO(3)$-bundles. (See \cite[Theorem 1.1 (1)]{fu7} and \cite{takagi}.) 
\par
Let $(M_i,\mathcal E_i)$ be a pair such that:
\[
  \partial(M_1,\mathcal E_1) = (\Sigma, \mathcal F)= -\partial(M_2,\mathcal E_2).
\]
Therefore, we can glue these two pairs to form a closed manifold $M$ and a $G$-bundle $\mathcal E$ over $M$. Conjecture \ref{conj84} asserts that there are pairs $(\widetilde L_{(M_1,\mathcal E_1)},b_1)$ and $(\widetilde L_{(M_2,\mathcal E_2)},b_2)$ of the objects of ${\frak Fuk}^G(X\backslash D,\omega)$.

\begin{conjecture}\label{cong8585}
	The Lagrangian Floer homology
	\[
	  HF_G((\widetilde L_{(M_1,\mathcal E_1)},b_1),(\widetilde L_{(M_2,\mathcal E_2)},b_2);X \setminus D)
	\]
	is an imvariant of $(M,\mathcal E_M)$.
\end{conjecture}

\section{Admissible Bundles and Instanton Floer Homology}
\label{admissible}

Let $E$ be a hermitian vector bundle of rank $N$ over a 3-manifold $M$. Then the vector bundle $E$ is determined up to isomorphism by its first Chern class. The pair $(M,E)$ is called an {\it admissible pair} if there is an oriented embedded surface $S$ in $M$ such that the pairing of $c_1(E)$ and the fundamental class of $S$ is coprime to $N$. We will also write $\mathcal E$ for the ${\rm PU}(N)$-bundle associated to $E$. Note that $M$ in this case is not an integral homology sphere. Floer's instanton homology can be extended to the case that $(M,\mathcal E)$ arises from an admissible pair \cite{fl3,KM:YAFT}. We will write $I^*(M,\mathcal E)$ for this version of instanton Floer homology. (See \cite[Section 3.1]{DX:suture} for a review of the general properties of $I^*(M,\mathcal E)$.) The proposal of the previous section to define symplectic instanton Floer homology can be also specialized to admissible pairs. We shall keep using the notation $I^*_{symp}(M,\mathcal E)$ to denote this conjectural invariant. There is yet another approach to define symplectic instanton Floer homology of $(M,\mathcal E)$ in this context, temporarily denoted by $\widetilde I^*_{symp}(M,\mathcal E)$, which avoids the technical difficulties of equivariant Floer homology in divisor complements. The current section concerns the relationship between the invariants $I^*_{symp}(M,\mathcal E)$ and $\widetilde I^*_{symp}(M,\mathcal E)$.

The definition of $\widetilde I^*_{symp}(M,\mathcal E)$ follows a similar route as $I^*_{symp}(M,\mathcal E)$. Suppose $F$ is a hermitian vector bundle of rank $N$ over an oriented Riemann surface $\Sigma$ such that the evaluation of $c_1(F)$ is coprime to $N$. Then the pair $(\Sigma,F)$ is called an admissible pair. Let $\mathcal F$ be the ${\rm PU}(2)$-bundle associated to $F$. Then the moduli space of flat connections $R(\Sigma,\mathcal F)$, defined in Section \ref{sec:extendedmoduli}, is a smooth K\"ahler manifold for this choice of $\mathcal F$. Let $M$ be a 3-manifold with boundary $\Sigma$ and $E$ be a hermitian vector bundle on $M$ extending $F$. Then we define ${\widetilde {\mathcal R}}(M,\mathcal E)$ to be the space of all elements of $R(\Sigma,\mathcal F)$ represented by flat connections on $\mathcal F$ that can be extended to $\mathcal E$. This space can be perturbed to an immersed  Lagrangian submanifold of $R(\Sigma,\mathcal F)$ which we still denote by ${\widetilde {\mathcal R}}(M,\mathcal E)$ \cite{He}. The moduli space of solutions to the mixed equation can be also used to define a bounding cochain $b_{(M,\mathcal E)}$ for this Lagrangian \cite[Theorem 1.1 (1)]{fu7}. Therefore, $({\widetilde {\mathcal R}}(M,\mathcal E),b_{(M,\mathcal E)})$ defines an object of ${\mathfrak {Fuk}}(R(\Sigma,\mathcal F))$. 

Next, let $(M,E)$ be an admissible pair. There is an embedded Riemann surface $\Sigma$ in $M$ such that removing $\Sigma$ from $M$ gives a disconnected manifold, and the pair given by $\Sigma$ and $F:=E|_{\Sigma}$ is admissible. Let $M_1$ and $M_2$ be the closure of the connected components of $M\backslash \Sigma$ and $E_i:=E|_{M_i}$. We can assume that the Lagrangians ${\widetilde {\mathcal R}}(M_1,\mathcal E_1)$ and ${\widetilde {\mathcal R}}(M_2,\mathcal E_2)$ have clean intersection by applying holonomy perturbations to one of them. Then $\widetilde I^*_{symp}(M,\mathcal E)$ is defined to be the Lagrangian Floer homology of the two elements $({\widetilde {\mathcal R}}(M_1,\mathcal E_1),b_{(M_1,\mathcal E_1)})$ and $({\widetilde {\mathcal R}}(M_2,\mathcal E_2),b_{(M_2,\mathcal E_2)})$. This Lagrangian Floer homology is independent of the choice of $\Sigma$. 

\begin{conjecture} \label{admissible-AF-synp-quo}
	For an admissible pair $(M,E)$, the vector spaces $I^*(M,E)$, $I_{symp}^*(M,E)$ and ${\widetilde I}_{symp}^*(M,E)$
	are isomorphic to each other.
\end{conjecture}

The part of the above conjecture about the relationship between  $I^*(M,E)$ and ${\widetilde I}_{symp}^*(M,E)$ is what we previously referred as the {\it $\SO(3)$-analogue of the Atiyah-Floer conjecture}. In the case that the Lagrangians involved in the definition of ${\widetilde I}_{symp}^*(M,E)$ are embedded, the $\SO(3)$-Atiyah-Floer conjecture is addressed in \cite{DFL}. The more general case will be treated in another forthcoming paper following the strategy proposed in \cite{fu7}. In the following, we discuss some general results in symplectic Floer homology which are related to the part of Conjecture \ref{admissible-AF-synp-quo} about the existence of isomorphism between $I_{symp}^*(M,E)$ and ${\widetilde I}_{symp}^*(M,E)$. Once the definition of the invariant $I_{symp}^*(M,E)$ is fully developed, we hope that these general results give a proof for this part of the above conjecture.

We firstly need an extension of the category $\frak {Fuk}(X,\omega,\fL)$ for a clean collection of immersed Lagrangian 
submanifolds $\fL$ in a symplectic manifold $(X,\omega)$. Suppose $\fb\in H^{{\rm even}}(X;\Lambda_0)$ with $\frak b \equiv 0 \mod T^{\epsilon}$.\footnote{The condition $\frak b \equiv 0 \mod T^{\epsilon}$ is not necessary. However, we need a slightly delicate argument to prove the convergence of operators. See, for example, \cite[Definition 17.8]{fooospectr}. For our application in this paper it suffices to consider the case when this extra condition is satisfied.} Then the  $A_{\infty}$ operations $\fm_k$ associated to $\fL$ can be deformed by $\frak b$ to $\frak m^{\frak b}_k$ as in \cite[Definition 3.8.38]{fooobook}. Such deformations of the $A_\infty$ structure of $\frak {Fuk}(X,\omega,\fL)$ are called {\it Lagrangian Floer theory with bulk deformation}. Roughly speaking, we deform $\frak m_k$ to $\frak m^{\frak b}_k$ using the holomorphic disks which hit a cycle that is Poincar\'e dual to $\frak b$. Bounding cochains of this deformed structure are also defined in the same way as in \eqref{formMC}. Consequently, there is a (filtered) $A_{\infty}$-category $\frak {Fuk}(X,\omega,\fL,\fb)$ whose objects are pairs of a Lagrangian $L \in \frak \fL$ and a bounding cochain $b$ with respect to the $\frak b$-deformed (filtered) $A_{\infty}$ structure \cite[Definition 3.8.38]{fooobook}.

\begin{shitu}\label{situ55}
	Suppose a Hamiltonian action of a Lie group $G$ on a symplectic manifold $(X,\omega)$ is given.
	Let $\mu : X \to \frak g^*$ denote the moment map of this action. Let the action of $G$ on $\mu^{-1}(0)$ is free. Then
	the quotient $Y=\mu^{-1}(0)/G$ is a symplectic manifold with a symplectic form $\overline \omega$ \cite{MaW}.
	Let $\frak L$ be a clean collection of $G$-equivariant immersed Lagrangian submanifolds. 
	For each $(\widetilde L,\iota_{\widetilde L})\in \fL$, we assume that the $G$ action on $\widetilde L$ is free and
	$\iota_{\widetilde L}(\widetilde L) \subset \mu^{-1}(0)$.
	Then $\overline L := (\widetilde L/G,[i_{\widetilde L}])$ is an immersed Lagrangian submanifold of $Y$.
	The collection of all such immersed Lagrangians of $Y$ is denoted by $\overline{\fL}$. 
	Finally we assume that the following Lagrangian:
	\begin{equation}\label{quotcorr}
		\{(x,y) \in X \times Y \mid x \in \mu^{-1}(0), y = [x]\}
	\end{equation}
	is spin.
\end{shitu}

\begin{theorem}\label{quotthem}
	There exists $\frak b\in H^{{\rm even}}(Y;\Lambda_0)$ such that the two filtered $A_{\infty}$-categories
	$\frak{Fuk}^G(X,\omega,\fL)$ and $\frak{Fuk}(Y,\overline \omega,\overline{\fL},\frak b)$ are homotopy 
	equivalent\footnote{See \cite[Definition 8.5]{fu6}.}. If \eqref{quotcorr} is a monotone Lagrangian
	with minimal Maslov number $>2$, then $\fb$ can be chosen to be zero.
\end{theorem}

\begin{remark}
	The element $\frak b$ in Theorem \ref{quotthem} is related to the {\it quantum Kirwan map} 
	introduced by Woodward in \cite{Wo10,Wo2,Wo3}. Theorem \ref{quotthem} is also closely related to the results 
	of Tian and Xu, written or announced in a series of papers \cite{TX}. 
	Both the works of Woodward and Tian-Xu (as well as various other related works such as \cite{GW})
	are based on the study of gauged sigma model \cite{Mu, CGS}.
	On the other hand, the second author's proof, which will appear in \cite{FuFu6}, 
	uses equivariant Kuranishi structures
	and relies on the idea of employing Lagrangian correspondence and cobordism arguments in a similar way 
	as in \cite{fu3,LL}. We were informed by Max Lipyanskiy that he had similar ideas
	to use Lagrangian correspondences and cobordism arguments instead of gauged sigma model.
\end{remark}


\begin{conjecture}\label{cong107}
	Suppose $(X,\omega)$ and $\fL$ are given as in Situation \ref{situ55}. Moreover, assume that 
	there exists a $G$-invariant normal crossing divisor $D \subset X \setminus \mu^{-1}(0)$
	such that $X \setminus D$ is monotone. Let \eqref{quotcorr} be a monotone Lagrangian 
	submanifold of $(X \setminus D) \times Y$. 
	Then the category of filtered $A_{\infty}$ category $\frak{Fuk}^G(X \setminus D,\omega,\fL)$ 
	 is homotopy equivalent to $\frak{Fuk}(Y,\overline \omega,\overline \fL)$.
\end{conjecture}

\begin{remark}
	The above generalization of Theorem \ref{quotthem} is related to Conjecture \ref{admissible-AF-synp-quo}. 
	By picking $X=\widehat{R}(\Sigma,\mathcal N,\mathcal F,p,{\epsilon})$, 
	this conjecture implies the predicted relationship between 
	$I_{symp}^*(M,E)$ and ${\widetilde I}_{symp}^*(M,E)$ in Conjecture \ref{admissible-AF-synp-quo}. 
	The main difficulty with this 
	conjecture is to define $\frak{Fuk}^G(X \setminus D,\omega,\fL)$  for the case that $D$ is a normal crossing divisor. 
	Existence of this $A_\infty$-category in the special case that $D$ is a 
	smooth divisor is the content of Theorem \ref{cat-div-comp}. 
	Combination of the techniques used in verifying Theorems \ref{cat-div-comp} and \ref{quotthem} 
	prove the above conjecture in the special case that $D$ is a smooth divisor.
\end{remark}

\bibliographystyle{amsalpha}
\bibliography{references}

\end{document}